\documentclass[11pt]{amsart}
\usepackage{amssymb, amsmath, amsthm, mathrsfs, amsopn}
\usepackage{amsrefs}

\usepackage[a4paper, centering]{geometry}
\def\DM{{\mathcal{DM}^{\infty}}}

\newcommand*{\DMloc}[1][\Omega]{\mathcal{DM}^{\infty}_{{\rm loc}}{(#1)}}
\newcommand*{\BVLloc}[1][\Omega]{BV_{{\rm loc}}(#1)\cap L^{\infty}_{{\rm loc}}{(#1)}}
\DeclareMathOperator{\Tr}{Tr}
\newcommand*{\Trace}[3][\pm]{\Tr^{#1}(#2, #3)}
\newcommand*{\Trp}[2]{\Trace[+]{#1}{#2}}
\newcommand*{\Trm}[2]{\Trace[-]{#1}{#2}}
\newcommand*{\chiut}[1][]{\chi^{#1}_{\{u > t\}}}
\newcommand*{\jump}[1]{\Theta_{#1}}

\def\nuint{\widetilde{\nu}}

\newcommand{\A}{\boldsymbol{A}}

\def\R{\mathbb{R}}

\def\N{\mathbb{N}}

\DeclareMathOperator{\Div}{div}

\newcommand{\medint}{-\kern  -,375cm\int}
\newcommand{\medintinrigo}{-\kern  -,315cm\int}

\newcommand{\eps}{\varepsilon}
 \newcommand{\hh}{{\mathcal H}^{N-1}}

\newcommand{\LLN}{{\mathcal L}^N}

\newcommand{\M}[1]{\mathcal{#1}}    
\renewcommand{\H}{\M{H}}

\newcommand{\res}{\mathop{\hbox{\vrule height 7pt width .5pt depth 0pt
\vrule height .5pt width 6pt depth 0pt}}\nolimits} 

\def\pscal#1#2{\left\langle #1\,,\, #2 \right\rangle}
\DeclareMathOperator{\sign}{sign}
\def\ut{\widetilde{u}}
\def\polar{\theta}

\long\def\taglio#1{}

\newtheorem{definition}{Definition}[section]

\newtheorem{lemma}[definition]{Lemma}

\newtheorem{theorem}[definition]{Theorem}

\newtheorem{proposition}[definition]{Proposition}

\newtheorem{corollary}[definition]{Corollary}
\theoremstyle{remark}
\newtheorem{remark}[definition]{Remark}
\newtheorem{example}[definition]{Example}

\makeatletter
\def\@settitle{\begin{center}%
		\baselineskip14\p@\relax
		\bfseries
		\uppercasenonmath\@title
		\@title
		\ifx\@subtitle\@empty\else
		\\[5ex]
		\normalsize\mdseries\@subtitle
		\fi
	\end{center}%
}
\def\subtitle#1{\gdef\@subtitle{#1}}
\def\@subtitle{}
\makeatother

\begin{document}
\title[Gauss--Green theorem]
{Anzellotti's pairing theory \\ and the Gauss--Green theorem}

\author[G.~Crasta]{Graziano Crasta}
\address{Dipartimento di Matematica ``G.\ Castelnuovo'', Univ.\ di Roma I\\
	P.le A.\ Moro 2 -- I-00185 Roma (Italy)}
\email{crasta@mat.uniroma1.it}
\author[V.~De Cicco]{Virginia De Cicco}
\address{Dipartimento di Scienze di Base  e Applicate per l'Ingegneria, Univ.\ di Roma I\\
	Via A.\ Scarpa 10 -- I-00185 Roma (Italy)}
\email{virginia.decicco@sbai.uniroma1.it}

\date{October 10, 2018}

\begin{abstract}
In this paper we obtain a very general Gauss-Green formula 
for weakly differentiable functions 
and sets of finite perimeter.
This result is obtained by revisiting Anzellotti's pairing theory and by characterizing the measure 
pairing \((\A, Du)\) when
\(\A\) is a bounded divergence measure vector field and
\(u\) is a bounded function of bounded variation. 
\end{abstract}

\maketitle

\section{Introduction}

In the pioneering paper \cite{Anz},  Anzellotti established a pairing theory between 
weakly differentiable vector fields and \(BV\) functions.
Among other applications that will be mentioned below, this theory can be used to extend the validity of the
Gauss--Green formula to such vector fields and to 
non smooth domains.

As a means of comparison,
there are mainly two kinds of generalizations
of the Gauss--Green formula.
On one hand, 
one may consider
weakly differentiable vector fields but fairly regular
(e.g.\ Lipschitz) domains, see e.g.\ \cite{Cas}.
On the other hand, De Giorgi and Federer
consider fairly regular vector fields
and sets of finite perimeter
(see e.g.\ \cite[Theorem~3.36]{AFP}).
Other generalizations deal with weakly differentiable vector fields
and non-smooth domains,
see e.g.\ \cite{ChenFrid,ChTo,ChToZi}.
We mention also
\cite{ComiPayne,LeoSar} for some recent contributions on the subject.

In this paper we will prove a Gauss--Green formula 
valid for both
weakly differentiable vector fields and sets of finite perimeter.
This unifying result is obtained by revisiting Anzellotti's pairing theory
in the general case of 
divergence measure vector fields and $BV$ functions.
The core of the work 
is the characterization
of the normal traces of these vector fields 
and the analysis of the singular part of the pairing measure.
This will allow us to establish some nice formulas (coarea, chain rule, Leibnitz rule) for the pairing
and, eventually, to prove our general Gauss--Green formula.
We mention that, with our approach,
no approximation step
with smooth fields or smooth subdomains, 
in the spirit of \cite{Anz2} and \cite{ChenFrid,ChToZi,CoTo},
is needed.
On top of that, our feeling is that
the approximation with smooth fields may not work in our framework
(see the discussion before Proposition~\ref{p:approx}).

\smallskip

Let us describe in more detail the functional setting of the problem.
Let $\DM$ denote the class of
bounded divergence measure vector fields \(\A\colon\R^N \to \R^N\), 
i.e.\ the vector fields with the properties 
\(\A\in L^\infty\) and \(\Div \A\) is a finite Radon measure.
If \(\A\in\DM\) and \(u\) is a function of bounded variation with precise representative \(u^*\), 
then the distribution \((\A, Du)\), defined by
\begin{equation}\label{f:pintro}
\pscal{(\A, Du)}{\varphi} :=
-\int_{\R^N} u^*\varphi\, d \Div \A - \int_{\R^N} u \, \A\cdot \nabla\varphi\, dx,
\qquad \varphi\in C^\infty_c(\R^N)
\end{equation}
is a Radon measure in \(\R^N\), absolutely continuous with respect to \(|Du|\).
This fact has been proved by Anzellotti in \cite{Anz} for several combinations
of \(\A\) and \(u\)
(for instance $\Div\A\in L^1$ or $u$
a $BV$ continuous function), 
excluding the general case of 
$\A\in \DM$ and $u\in BV$.
Indeed, at that time, it was not clear how the discontinuities of \(u\)
interact with the discontinuities of the vector field \(\A\).
The pairing \eqref{f:pintro} has been defined in the general setting
by Chen--Frid
in the celebrated paper \cite{ChenFrid},
where the authors also characterized the absolutely continuous part
of the measure \((\A, Du)\) as \(\A \cdot \nabla u\).
Nevertheless, they have not characterized the singular part of the measure,
and, as far as we know, this problem has remained unsolved,
at least in this general setting.

\smallskip
On the other hand, the pairing 
in its full generality
has been revealed
a fundamental tool 
in several contexts. 
We cite, for example,
\cite{ChFr1,ChenFrid,ChTo2,ChTo,ChToZi,CDC2,Pan1}
for applications in the theory of hyperbolic systems of conservation and balance laws, 
and \cite{AmbCriMan} for the case of vector fields induced by functions of bounded deformation,
with the aim of extending the
Ambrosio--DiPerna--Lions theory of the transport equations (see also \cite{ADM}).

The divergence measure vector fields play a crucial role also in the theory of capillarity and 
in the study of the Prescribed Mean Curvature problem (see e.g.\ \cite{LeoSar,LeoSar2} and the references therein), 
and in the context of continuum mechanics 
(see e.g.\ \cite{DGMM,Silh,Schu}). 

Another field of application is related to
the Dirichlet problem for equations involving the $1$--Laplacian operator (see \cite{K1,HI,ABCM,Cas,DeGiOlPe,SchSch,SchSch2}). 
The interest in this setting comes out  from an optimal design problem, in the theory of torsion and from the level set formulation of the Inverse Mean Curvature Flow.  
To deal with the $1$--Laplacian $\Delta_1u:=\Div\left(\frac{Du}{|Du|}\right)$, 
the main difficulty is to define the quotient $\frac{Du}{|Du|}$, being $Du$ a Radon measure. 
This difficulty has been
overcome 
in \cite{ABCM,AVCM} 
through Anzellotti's theory of pairings. 
Namely, the role of this quotient is played by
a vector field $\A\in \DM$ such that $\|\A\|_\infty\le 1$ and $(\A, Du)=|Du|$.

Finally, in some lower semicontinuity problems for integral functionals defined in Sobolev spaces and in $BV$, 
the vector fields with measure--derivative occurred as natural dependence  of the integrand with respect to the spatial variable (see \cite{BouDM,DCFV2,dcl}).
To this end, we address the reader to the forthcoming paper \cite{CDC4},
where the authors introduce a nonlinear version of the pairing suitable for
applications to semicontinuity problems.

\medskip
Let us now describe in more detail the results proved in this paper.

Our first aim is to characterize the measure \((\A, Du)\) in the general case
\(\A\in\DM\) 
and
\(u\in BV\).
As we have already recalled above,
the absolutely continuous part of 
\((\A, Du)\) has been characterized
in \cite{ChenFrid} as \(\A\cdot \nabla u\),
hence only the jump and the Cantor parts
have to be studied.

The analysis of the jump part of the pairing
requires, in particular,
a detailed study of the normal traces of $u\A$ on
an oriented countably $\H^{N-1}$-rectifiable set $\Sigma$.
Following the arguments in \cite{AmbCriMan}, 
in Proposition~\ref{p:tracesb} below
we will prove that,
if \(\A\in\DM\) and \(u\in BV\cap L^\infty\),
then \(u\A \in \DM\) and the normal traces of \(u\A\) on \(\Sigma\) are given by
\[
\Trace{u\A}{\Sigma} = u^\pm \Trace{\A}{\Sigma},
\qquad \hh-\text{a.e.\ in}\ \Sigma. 	
\]

This allows us to give a precise description of the jump part $(\A, Du)^j$ of the measure $(\A, Du)$ in terms of the trace of $u$ and the normal trace of $\A$. 

Under the additional assumption
$|D^c u|(S_{\A}) = 0$,
where $D^c u$ is the Cantor part of $Du$ and
$S_{\A}$ is the approximate
discontinuity set of $\A$,
we are able to
give a representation formula for the 
Cantor part $(\A, Du)^c$ of the pairing measure.
In Remark~\ref{r:ipoCantor}
we will discuss some cases of interest
where this condition is satisfied.

In conclusion,
in Section~3
we will prove that the measure \((\A, Du)\) admits the following decomposition:
\begin{itemize}
	\item[(i)]
	absolutely continuous part: 
	\((\A, Du)^a = \A \cdot \nabla u\, \LLN\);
	
	\item[(ii)]
	jump part:
	\(\displaystyle
	(\A, Du)^j = 
	\frac{\Trp{\A}{J_u}+\Trm{\A}{J_u}}{2}
	\, (u^+-u^-) \, \hh \res J_u
	\);	

	\item[(iii)]
	Cantor part:
	if $|D^c u|(S_{\A}) = 0$, then
	\((\A, Du)^c = \widetilde{\A} \cdot D^c u\),
\end{itemize}
where $\widetilde{\A}$ is the approximate limit of $\A$
defined in $\R^N\setminus S_{\A}$.

\smallskip

In Section~\ref{s:form}, by using the above decomposition, 
we will be able to describe the Radon--Nikod\'ym derivative of the measure $(\A, Du)$
with respect to $|Du|$,
and to obtain a very general coarea formula.
As a consequence,
we will prove the Leibniz formula
for $(\A, D(uv))$ and $(v\A, Du)$.
Then, we will prove an approximation result 
by regular vector fields
and a semicontinuity result.

\smallskip

Finally, in Section~\ref{s:GG},
exploiting the formulas proved in Section~\ref{s:form},
we will prove our generalized Gauss-Green formula: 
if $\A\in\DM$, $u\in BV\cap L^\infty$,
and $E\subset\R^N$ is a bounded set with finite perimeter,
then
\begin{gather}
\int_{E^1} u^* \, d\Div\A + \int_{E^1} (\A, Du) = -
\int_{\partial ^*E} u^+\, \Trp{\A}{\partial^*E} \ d\mathcal H^{N-1}\,, \label{GreenIBINTRO}
\\
\int_{E^1\cup \partial ^*E} u^* \, d\Div\A + \int_{E^1\cup \partial ^*E} (\A, Du) = -
\int_{\partial ^*E} u^-\, \Trm{\A}{\partial^*E} \ d\mathcal H^{N-1}\,, \label{GreenIB2INTRO}
\end{gather}
where $E^1$ is the measure theoretic interior of $E$, \(\partial^* E\) is the reduced boundary of $E$
and \(\partial^* E\) is oriented
with respect to the interior unit normal vector.

As we have already underlined in this introduction,
a number of Gauss--Green formulas that can be found in the literature
are particular cases of \eqref{GreenIBINTRO}
and \eqref{GreenIB2INTRO}.

For example, the case $u\equiv 1$ has been considered
in the classical De Giorgi--Federer formula with $\A$ a regular vector field 
(see e.g.\ \cite[Theorem~3.36]{AFP}),
by Vol'pert \cite{vol,vol1} for $\A\in BV(\Omega, \R^N)$
and finally by Chen--Torres--Ziemer \cite{ChToZi} in the general case $\A\in\DM$.

The case of a non-constant $u$ has been considered
by Anzellotti \cite{Anz2} if $\Div\A\in L^1$,
by Comi--Payne \cite{ComiPayne} if $u$ is a locally Lipschitz function,
and by Leonardi--Saracco if $\A\in \DM\cap C^0$ (with some additional
conditions on $E$).

\section{Preliminaries}

In this paper we mainly follow the notation of \cite[Chapter~3]{AFP}.

In the following \(\Omega\) will always denote a nonempty open subset of \(\R^N\).

Let \(u\in L^1_{{\rm loc}}(\Omega)\).
We say that \(u\) has an approximate limit at $x_{0}\in\Omega$ 
if exists \(z\in\R\) such that
\begin{equation}
\label{f:apcont}
\lim_{r\rightarrow0^{+}}\frac{1}{\LLN\left(  B_r(x_0)\right)}\int_{B_r\left(  x_{0}\right)
}\left|  u(x)  -z  \right|  \,dx=0.
\end{equation}
The set \(S_u\subset\Omega\) of points where this property does not hold is called
the approximate discontinuity set of \(u\).
For every \(x_0\in\Omega\setminus S_u\) the number \(z\), uniquely determined by
\eqref{f:apcont}, is called the approximate limit of \(u\) at \(x_0\)
and denoted by \(\ut(x_0)\).

We say that \(x_0\in\Omega\) is an approximate jump point of \(u\) if
there exist \(a,b\in\R\) and a unit vector \(\nu\in\R^n\) such that \(a\neq b\)
and
\begin{equation}\label{f:disc}
\begin{gathered}
\lim_{r \to 0^+} \frac{1}{\LLN(B_r^+(x_0))}
\int_{B_r^+(x_0)} |u(y) - a|\, dy = 0,
\\
\lim_{r \to 0^+} \frac{1}{\LLN(B_r^-(x_0))}
\int_{B_r^-(x_0)} |u(y) - b|\, dy = 0,
\end{gathered}
\end{equation}
where \(B_r^\pm(x_0) := \{y\in B_r(x_0):\ \pm (y-x_0)\cdot \nu > 0\}\).
The triplet \((a,b,\nu)\), uniquely determined by \eqref{f:disc} 
up to a permutation
of \((a,b)\) and a change of sign of \(\nu\),
is denoted by \((u^+(x_0), u^-(x_0), \nu_u(x_0))\).
The set of approximate jump points of \(u\) will be denoted by \(J_u\).

The notions of approximate discontinuity set, approximate limit and approximate jump point 
can be obviously extended to the vectorial case
(see \cite[\S 3.6]{AFP}).

In the following we shall always extend the functions $u^\pm$ to
\(\Omega\setminus(S_u\setminus J_u)\) by setting
\[
u^\pm \equiv \widetilde{u}\ \text{in}\ \Omega\setminus S_u.
\]
 
In some occasions it will be useful to choose the orientation of \(\nu\)
in such a way that \(u^- < u^+\) in $J_u$.
These particular choices of $u^-$ and $u^+$ will be called
the approximate lower limit and the approximate upper limit of $u$ respectively.

Here and in the following we will denote by \(\rho \in C^\infty_c(\R^N)\) a
symmetric convolution kernel with support in the unit ball,
and by \(\rho_{\varepsilon}(x) := \varepsilon^{-N} \rho(x/\varepsilon)\).

In the sequel we will use often the following result
(see \cite[Proposition~3.64(b)]{AFP}).

\begin{proposition}\label{lolol}
Let \(u\in L^1_{{\rm loc}}(\Omega)\)
and define
\[
u_{\varepsilon}\left(  x\right) = \rho_{\varepsilon} \ast u (x) :=\int_{\Omega}\rho_{\varepsilon}\left(
x-y\right)  \,u\left(  y\right)  \,dy.
\]
If \(x_0\in\Omega\setminus S_u\), then \(u_\eps(x_0) \to \ut(x_0)\)
as \(\eps\to 0^+\).
\end{proposition}

\subsection{Functions of bounded variation and sets of finite perimeter}

We say that \(u\in L^1(\Omega)\) is a function of bounded variation in \(\Omega\)
if the distributional derivative \(Du\) of \(u\) is a finite Radon measure in \(\Omega\).
The vector space of all functions of bounded variation in \(\Omega\)
will be denoted by \(BV(\Omega)\).
Moreover, we will denote by \(BV_{{\rm loc}}(\Omega)\) the set of functions
\(u\in L^1_{{\rm loc}}(\Omega)\) that belongs to 
\(BV(A)\) for every open set \(A\Subset\Omega\)
(i.e., the closure \(\overline{A}\) of \(A\) is a compact
subset of \(\Omega\)).

If \(u\in BV(\Omega)\), then \(Du\) can be decomposed as
the sum of the absolutely continuous and the singular part with respect
to the Lebesgue measure, i.e.\
\[
Du = D^a u + D^s u,
\qquad D^a u = \nabla u \, \LLN,
\]
where \(\nabla u\) is the approximate gradient of \(u\),
defined \(\LLN\)-a.e.\ in \(\Omega\)
(see \cite[Section~3.9]{AFP}).
On the other hand, 
the jump set $J_u$ is countably $\H^{N-1}$--rectifiable,
$\H^{N-1}(S_u  \setminus J_u) = 0$
(see \cite[Definition~2.57 and Theorem~3.78]{AFP}), and
the singular part \(D^s u\) can be further decomposed
as the sum of its Cantor and jump part, i.e.
\[
D^s u = D^c u + D^j u,
\qquad
D^c u := D^s u \res (\Omega\setminus S_u),
\quad
D^j u := D^s u \res J_u,
\]
where the symbol \(\mu\res B\) denotes the restriction of the measure \(\mu\)
to the set \(B\).
We will denote by \(D^d u := D^a u + D^c u\) the diffuse part of the measure \(Du\).

The precise representative $u^*$ of $u$ is defined in
$\Omega\setminus(S_u\setminus J_u)$
(hence $\H^{N-1}$-a.e.\ in $\Omega$)
as $\ut(x)$ when  $x\in \Omega\setminus S_u$, and
as $(u^+(x)+u^-(x)) / 2$ when $x\in J_u$.
The mollified functions $u_\varepsilon$ pointwise converge to
$u^*$ in its domain
(see \cite[Corollary~3.80]{AFP}).

In the following,
we will denote by $\polar_u\colon\Omega\to S^{N-1}$ the Radon--Nikod\'ym derivative
of $Du$ with respect to $|Du|$, i.e.\
the unique function $\polar_u \in L^1(\Omega, |Du|)^N$ such that the polar
decomposition $Du = \polar_u\, |Du|$ holds.
Since all parts of the derivative of $u$ are mutually singular, we have
\[
D^a u = \polar_u\, |D^a u|,
\quad
D^j u = \polar_u\, |D^j u|,
\quad
D^c u = \polar_u\, |D^c u|
\]
as well.
In particular $\polar_u(x) = \nabla u(x) / |\nabla u(x)|$ for
$\LLN$-a.e.\ $x\in\Omega$ such that $\nabla u(x) \neq 0$
and
$\polar_u(x) = \sign(u^+(x) - u^-(x))\, \nu_u(x)$
for $\H^{N-1}$-a.e.\ $x\in J_u$.

Let \(E\) be an \(\LLN\)-measurable subset of \(\R^N\).
For every open set \(\Omega\subset\R^N\) the perimeter \(P(E, \Omega)\)
is defined by
\[
P(E, \Omega) := \sup\left\{
\int_E \Div \varphi\, dx:\ \varphi\in C^1_c(\Omega, \R^N),\ \|\varphi\|_\infty\leq 1
\right\}.
\]
We say that \(E\) is of finite perimeter in \(\Omega\) if \(P(E, \Omega) < +\infty\).

Denoting by \(\chi_E\) the characteristic function of \(E\),
if \(E\) is a set of finite perimeter in \(\Omega\), then
\(D\chi_E\) is a finite Radon measure in \(\Omega\) and
\(P(E,\Omega) = |D\chi_E|(\Omega)\).

If \(\Omega\subset\R^N\) is the largest open set such that \(E\)
is locally of finite perimeter in \(\Omega\),
we call reduced boundary \(\partial^* E\) of \(E\) the set of all points
\(x\in \Omega\) in the support of \(|D\chi_E|\) such that the limit
\[
\nuint_E(x) := \lim_{\rho\to 0^+} \frac{D\chi_E(B_\rho(x))}{|D\chi_E|(B_\rho(x))}
\]
exists in \(\R^N\) and satisfies \(|\nuint_E(x)| = 1\).
The function \(\nuint\colon\partial^* E\to S^{N-1}\) is called
the measure theoretic unit interior normal to \(E\).

A fundamental result of De Giorgi (see \cite[Theorem~3.59]{AFP}) states that
\(\partial^* E\) is countably \((N-1)\)-rectifiable
and \(|D\chi_E| = \hh\res \partial^* E\).

Let \(E\) be an \(\LLN\)-measurable subset of \(\R^N\).
For every \(t\in [0,1]\) we denote by \(E^t\) the set
\[
E^t := \left\{x\in\R^N:\
\lim_{\rho\to 0^+} \frac{\LLN(E\cap B_\rho(x))}{\LLN(B_\rho(x))} = t\right\}
\]
of all points where \(E\) has density \(t\).
The sets \(E^0\), \(E^1\), \(\partial^e E := \R^N\setminus (E^0 \cup E^1)\) are called 
respectively the measure theoretic exterior, the measure theoretic interior and
the essential boundary of \(E\).
If \(E\) has finite perimeter in \(\Omega\), Federer's structure theorem
states that
\(\partial^* E\cap\Omega \subset E^{1/2} \subset \partial^e E\)
and \(\H^{N-1}(\Omega\setminus(E^0\cup \partial^e E \cup E^1)) = 0\)
(see \cite[Theorem~3.61]{AFP}).

\subsection{Capacity.}\label{not5}
In this section we recall the notion of $1$-capacity
and some results (Theorem~\ref{t:lecce} and Lemma~\ref{maso1})
that will be used in the proof of Proposition~\ref{approx2}.

Given an open set $A\subset\R^N$, the {\it $1$-capacity} of $A$ is defined by setting  
\begin{equation*}
C_1(A) : = \inf\left\{\int_{\R^N } |D\varphi|\,dx\ :\ \varphi\in W^{1,1}(\R^N),
\quad \varphi\geq 1\quad \LLN{\rm -a.e.\ on}\ A\right\}\,.
\end{equation*}
Then, the $1$-capacity of an arbitrary set $B\subset\R^N$ is given by
\begin{equation*}
C_1(B) := \inf\{C_1(A)\ :\ A\supseteq B,\ A\ {\rm open}\}\,.
\end{equation*}
It is well known that capacities and Hausdorff measure are closely related. In particular, we have
that for every Borel set $B\subset\R^N$
\[
C_1(B)=0\qquad\Longleftrightarrow\qquad
\H^{N-1}(B)=0\,.
\]

We recall that a function $u\colon\R^N\to\R$ is said $C_1$-quasi continuous
if for every $\eps>0$ there exists an open set $A$, with $C_1(A)<\eps$, such that
the restriction $u\res{A^c}$ is continuous on $A^c$;
$C_1$-quasi lower semicontinuous  and $C_1$-quasi upper semicontinuous
functions are defined similarly.

It is well known that every function $u\in W^{1,1}$
admits a $C_1$-quasi continuous representative 
that coincides $\H^{N-1}$-a.e.\ with $\ut$
(see \cite[Sections 9 and 10]{FZ}).
Moreover, to every $BV$-function $u$, it is possible to associate a $C_1$-quasi lower semicontinuous
and a $C_1$-quasi upper semicontinuous representative, as stated by the
following theorem (see \cite{CDLP}, Theorem~2.5).

\medskip

\begin{theorem}\label{t:lecce}
	For every function $u\in BV(\Omega)$, the approximate upper limit $u^+$ and the approximate
	lower limit $u^-$ are $C_1$-quasi upper semicontinuous and $C_1$-quasi lower semicontinuous,
	respectively.
\end{theorem}

In particular, if $B$ is a Borel subset of $\R^N$ with finite perimeter, then $\chi^{-}_B$
is $C_1$-quasi lower semicontinuous and $\chi^{+}_B$
is $C_1$-quasi upper semicontinuous.

We recall the following lemma which is an approximation result due to Dal Maso (see \cite{DM83},
Lemma 1.5 and \S 6).

\begin{lemma}\label{maso1}
	Let $u\colon\R^N\to[0,+\infty)$ be a $C_1$-quasi  lower semicontinuous function.
	Then there exists an increasing sequence of nonnegative
	functions $\{u_h\}\subseteq W^{1,1}(\R^N)$ such that, for every $h\in\N$,
	$u_h$ is approximately continuous $\H^{N-1}$-almost everywhere in $\R^N$ and
	$\ut_h(x)\to u(x)$, when $h\to+\infty$,  for $\H^{N-1}$-almost every $x\in\R^N$.
\end{lemma}

\subsection{Divergence--measure fields }
\label{ss:div}

We will denote by \(\DM(\Omega)\) the space of all
vector fields 
\(\A\in L^\infty(\Omega, \R^N)\)
whose divergence in the sense of distribution is a bounded Radon measure in \(\Omega\).
Similarly, \(\DMloc[\Omega]\) will denote the space of
all vector fields \(\A\in L^\infty_{{\rm loc}}(\Omega, \R^N)\)
whose divergence in the sense of distributions is a Radon measure in \(\Omega\). 
We set \(\DM = \DM(\R^N)\).

We recall that, if \(\A\in\DMloc[\Omega]\), then \(|\Div\A| \ll \hh\)
(see \cite[Proposition 3.1]{ChenFrid}).
As a consequence, the set
\begin{equation}\label{f:jump}
\jump{\A} 
:= \left\{
x\in\Omega:\
\limsup_{r \to 0+}
\frac{|\Div \A| (B_r(x))}{r^{N-1}} > 0
\right\},
\end{equation} 
is a Borel set, \(\sigma\)-finite with respect to \(\hh\),
and the measure \(\Div \A\) can be decomposed as
\[
\Div\A = \Div^a\A + \Div^c\A + \Div^j\A,
\]
where \(\Div^a\A\) is absolutely continuous with respect to \(\LLN\),
\(\Div^c\A (B) = 0\) for every set \(B\) with \(\hh(B) < +\infty\),
and
\[
\Div^j\A = f\, \hh\res\jump{\A}
\]
for some Borel function \(f\)
(see \cite[Proposition~2.5]{ADM}).

\subsection{Normal traces}
\label{distrtraces}
The traces of the normal component of the vector field \(\A\in \DMloc\)
can be defined as distributions
\(\Trace{\A}{\Sigma}\)
on every oriented countably \(\H^{N-1}\)--rectifiable set
\(\Sigma\subset\Omega\)
in the sense of Anzellotti
(see \cite{AmbCriMan,Anz,ChenFrid}).

More precisely, 
let us briefly recall the construction given in \cite{AmbCriMan} (see Propositions~3.2, 3.4 and Definition~3.3).
First of all, given a domain \(\Omega'\Subset\Omega\) of class \(C^1\), we define
the trace of the normal component of \(\A\) on \(\partial\Omega'\) as
a distribution as follows:
\begin{equation}\label{f:disttr}
\pscal{\Trace[]{\A}{\partial\Omega'}}{\varphi}
:= \int_{\Omega'} \A\cdot \nabla\varphi\, dx + \int_{\Omega'} \varphi\, d\Div\A,
\qquad
\forall\varphi\in C^\infty_c(\Omega).
\end{equation}
It turns out that this distribution is induced by an \(L^\infty\) function on \(\partial\Omega'\),
still denoted by \(\Trace[]{\A}{\partial\Omega'}\), and
\[
\|\Trace[]{\A}{\partial\Omega'}\|_{L^\infty(\partial\Omega')}
\leq \|\A\|_{L^\infty(\Omega')}.
\]

Since \(\Sigma\) is oriented and countably $\hh$--rectifiable,
we can find countably many \textsl{oriented} \(C^1\) hypersurfaces \(\Sigma_i\),
with classical normal \(\nu_{\Sigma_i}\),
and pairwise disjoint Borel sets \(N_i\subseteq \Sigma_i\)
such that \(\hh(\Sigma\setminus \bigcup_i N_i) = 0\).

Moreover, it is not restrictive to assume that, for every \(i\),  
there exist two open bounded sets \(\Omega_i, \Omega'_i\) with \(C^1\) boundary
and exterior normal vectors \(\nu_{\Omega_i}\) and \(\nu_{\Omega_i'}\) respectively,
such that
\(N_i\subseteq \partial\Omega_i \cap \partial\Omega'_i\)
and
\[
\nu_{\Sigma_i}(x) = \nu_{\Omega_i}(x) = -\nu_{\Omega'_i}(x)
\qquad \forall x\in N_i.
\]
At this point we choose, on \(\Sigma\), the orientation given by
\(\nu_{\Sigma}(x) := \nu_{\Sigma_i}(x)\)
\(\hh\)-a.e.\ on \(N_i\).

Using the localization property proved in \cite[Proposition 3.2]{AmbCriMan},
we can define
the normal traces of \(\A\) on \(\Sigma\) by
\[
\Trm{\A}{\Sigma} := \Tr(\A, \partial\Omega_i),
\quad
\Trp{\A}{\Sigma} := -\Tr(\A, \partial\Omega'_i),
\qquad
\hh-\text{a.e.\ on}\ N_i.
\]

These two normal traces belong to
\(L^{\infty}(\Sigma, \H^{N-1}\res\Sigma)\) (see \cite[Proposition 3.2]{AmbCriMan})
and
\begin{equation}\label{mmm}
\Div \A \res\Sigma =
\left[\Trp{\A}{\Sigma} - \Trm{\A}{\Sigma}\right]
\, {\mathcal H}^{N-1} \res\Sigma\,.
\end{equation}

\subsection{Anzellotti's pairing}
As in Anzellotti \cite{Anz} (see also \cite{ChenFrid}),
for every \(\A\in\DMloc\) and \(u\in\BVLloc\) we define the linear functional
\((\A, Du) \colon C^\infty_0(\Omega) \to \R\) by
\begin{equation}\label{f:pairing}
\pscal{(\A, Du)}{\varphi} :=
-\int_\Omega u^*\varphi\, d \Div \A - \int_\Omega u \, \A\cdot \nabla\varphi\, dx. 
\end{equation}
The distribution \((\A, Du)\) is a Radon measure in \(\Omega\),
absolutely continuous with respect to \(|Du|\)
(see \cite[Theorem 1.5]{Anz} and \cite[Theorem 3.2]{ChenFrid}),
hence the equation
\begin{equation}\label{f:anz}
\Div(u\A) = u^* \Div\A + (\A, Du)
\end{equation}
holds in the sense of measures in \(\Omega\)
(We remark that, in \cite{ChenFrid}, the measure \((\A, Du)\) is denoted
by \(\overline{\A\cdot Du}\).)
Furthermore, Chen and Frid in \cite{ChenFrid} proved that the absolutely continuous part
of this measure with respect to the Lebesgue measure is given by
\(
(\A, Du)^a = \A \cdot \nabla u\, \LLN
\).

\section{Characterization of Anzellotti's pairing}
\label{s:anz}

\begin{proposition}\label{p:tracesb}
	Let \(\A\in\DMloc\),  \(u\in\BVLloc\) and let \(\Sigma\subset\Omega\) 
	be a countably \(\hh\)--rectifiable set, oriented
	as in Section~\ref{distrtraces}.
	Then \(u\A\in\DMloc\) and the normal traces of \(u\A\) on \(\Sigma\) are given by
	\begin{equation}\label{f:trus}
		\Trace{u\A}{\Sigma} = 
		\begin{cases}
			u^\pm \Trace{\A}{\Sigma},
			& \hh-\text{a.e.\ in}\ J_u\cap\Sigma,\\
			\ut\, \Trace{\A}{\Sigma},
			& \hh-\text{a.e.\ in}\ \Sigma\setminus J_u. 	
		\end{cases}
	\end{equation}
\end{proposition}

\begin{proof}
The fact that \(u\A\in\DMloc\) has been proved
in \cite[Theorem~3.1]{ChenFrid}.

We will use the same notations of Section~\ref{distrtraces}.
It is not restrictive to assume that \(J_u\) is oriented with \(\nu_\Sigma\)
on \(J_u\cap \Sigma\).

Let us prove \eqref{f:trus} for \(\Tr^-\).
Let \(x\in \Sigma\) satisfy: 
\begin{itemize}
	\item[(a)] \(x\in(\Omega\setminus S_u)\cup J_u\), 
	\(x\in N_i\) for some \(i\), the set \(N_i\) has density \(1\) at \(x\), 
	and \(x\) is a Lebesgue point of \(\Trm{\A}{\Sigma}\)
	with respect to
	\(\hh\res\partial\Omega_i\);
	
	\item[(b)]
	\(|\Div\A| \res\Omega_i (B_\varepsilon(x)) = o(\varepsilon^{N-1})\) as \(\varepsilon\to 0\);
	
	\item[(c)]
	\(|\Div (u\A)| \res\Omega_i  (B_\varepsilon(x)) = o(\varepsilon^{N-1})\).
\end{itemize}
We remark that \(\hh\)-a.e.\ \(x\in \Sigma\) satisfies these conditions.
In particular, (a) is satisfied because 
\(\hh(S_u\setminus J_u) = 0\),
whereas 
(b) and (c) follow from \cite[Theorem~2.56 and (2.41)]{AFP}.

In order to simplify the notation,
in the following we set \(u^-(x) := \ut(x)\) if \(x\in \Omega\setminus S_u\).

Let us choose a function \(\varphi\in C^{\infty}_c(\R^N)\), with support contained
in \(B_1(0)\), such that \(0\leq \varphi \leq 1\).
For every \(\varepsilon > 0\) let \(\varphi_{\varepsilon}(y) := \varphi\left(\frac{y-x}{\varepsilon}\right)\).

By the very definition of normal trace, the following equality holds for
every \(\varepsilon > 0\) small enough:
\begin{equation}\label{f:tra}
\begin{split}
\frac{1}{\varepsilon^{N-1}} &
\int_{\partial\Omega_i} 
[\Tr(u\A, \partial\Omega_i) -
u^-(x) \Tr(\A, \partial\Omega_i)]
\, \varphi_{\varepsilon}(y)\, d\mathcal{H}^{N-1}(y)
\\ = {} &
\frac{1}{\varepsilon^{N-1}}
\int_{\Omega_i} \nabla\varphi_{\varepsilon}(y) \cdot [u(y)\A(y) - u^-(x)\A(y)]\, dy
\\ & + 
\frac{1}{\varepsilon^{N-1}}
\int_{\Omega_i} \varphi_{\varepsilon}(y) \, d[\Div(u\A) - u^-(x) \Div \A](y)\,.
\end{split}
\end{equation}
Using the change of variable \(z = (y-x)/\varepsilon\),
as \(\varepsilon \to 0\) the left hand side of this equality converges to
\[
[\Trm{u\A}{\Sigma}(x) - u^-(x)\Trm{\A}{\Sigma}] \int_{\Pi_x} \varphi(z)\, d\mathcal{H}^{N-1}(z)\,,
\]
where \(\Pi_x\) is the tangent plane to \(\Sigma_i\) at \(x\).
Clearly \(\varphi\) can be chosen in such a way that
\(\int_{\Pi_x} \varphi\, d\H^{N-1} > 0\).

In order to prove \eqref{f:trus} for \(\Tr^-\) it is then enough
to show that the two integrals \(I_1(\varepsilon)\) and \(I_2(\varepsilon)\)
on the right hand side of \eqref{f:tra} converge to \(0\)
as \(\varepsilon \to 0\).

With the change of variables \( z = (y-x) / \varepsilon\) 
we have that
\[
I_1(\varepsilon) =
\int_{\Omega_i^\varepsilon} [ u(x+\varepsilon z) - u^-(x)] \nabla\varphi(z) \cdot
\A(x+\eps z)
\, dz,
\]
where
\[
\Omega_i^\varepsilon := \frac{\Omega_i - x}{\varepsilon}.
\]
As \(\varepsilon\to 0\), these sets locally converge to the half space 
\(P_x := \{z\in\R^N:\ \pscal{z}{\nu(x)} < 0\}\),
hence
\[
\lim_{\varepsilon\to 0}
\int_{\Omega_i^\varepsilon\cap B_1} |u(x+\varepsilon z) - u^-(x)|\, dz  = 
\lim_{\varepsilon\to 0}
\int_{P_x \cap B_1} |u(x+\varepsilon z) - u^-(x)|\, dz  = 0
\]
(see \cite[Remark 3.85]{AFP}) so that
\[
|I_1(\varepsilon)| \leq
\|\A\|_{L^\infty(B_\eps(x))}\, \|\nabla\varphi\|_{\infty} 
\int_{\Omega_i^\varepsilon\cap B_1} |u(x+\varepsilon z) - u^-(x)|\, dz  
\to 0.
\]

\smallskip
From (b) we have that
\[
\lim_{\varepsilon\to 0} \frac{1}{\varepsilon^{N-1}}
\left|\int_{\Omega_i}
\varphi_{\varepsilon}(y)\, u^-(x)\, d \Div\A(y)
\right|
\leq
\limsup_{\varepsilon\to 0}|u^-(x)| \frac{|\Div \A| (B_\varepsilon(x))}{\varepsilon^{N-1}}
=0.
\]
In a similar way, using (c), we get
\[
\lim_{\varepsilon\to 0}
\frac{1}{\varepsilon^{N-1}}
\left|\int_{\Omega_i} \varphi_{\varepsilon} \, d\Div(u\A)\right| = 0,
\]
so that \(I_2(\varepsilon)\)
vanishes as \(\varepsilon\to 0\).

The proof of \eqref{f:tr} for \(\Tr^+\) is entirely similar.	
\end{proof}

Since $J_u$ is a countably $\H^{N-1}$-rectifiable set,
a straightforward consequence of Proposition~\ref{p:tracesb}
is the following result  
(see also \cite[Lemma 2.5]{GiMoPe}).

\begin{corollary}\label{c:traces}
	Let \(\A\in\DMloc\) and \(u\in\BVLloc\).
	Then \(u\A \in \DMloc\) and the normal traces of \(u\A\) on \(J_u\) are given by
	\begin{equation}\label{f:tr}
		\Trace{u\A}{J_u} = u^\pm \Trace{\A}{J_u},
		\qquad \hh-\text{a.e.\ in}\ J_u. 	
	\end{equation}
	In particular
	\begin{equation}\label{f:truA}
		\Div(u\A)\res J_u =
		\left[
		u^+ \Trp{\A}{J_u} - u^- \Trm{\A}{J_u}
		\right]\, \hh \res J_u.
	\end{equation}
\end{corollary}

We are now ready to prove the main decomposition theorem
for the pairing measure.
We observe that a more general result,
for unbounded $BV$ functions,
will be proved in Theorem~\ref{t:pairing2} below.

\begin{theorem}\label{t:pairing}
Let \(\A\in\DMloc\) and \(u\in\BVLloc\).
Then the measure \((\A, Du)\) admits the following decomposition:
\begin{itemize}
	\item[(i)]
	absolutely continuous part: 
	\((\A, Du)^a = \A \cdot \nabla u\, \LLN\);
	
	\item[(ii)]
	jump part:
	\(\displaystyle
	(\A, Du)^j = 
	\frac{\Trp{\A}{J_u}+\Trm{\A}{J_u}}{2}
	\, (u^+-u^-) \, \hh \res J_u
	\);

	\item[(iii)]
	diffuse part:
	if, in addition, 
	\begin{equation}\label{f:ipo}
	|D^c u| (S_{\A}) = 0,
	\end{equation}
	where \(S_{\A}\) is the approximate discontinuity set  of \(\A\),
	then
	\((\A, Du)^d = \widetilde{\A} \cdot D^d u\).
	
\end{itemize}
\end{theorem}

\begin{remark}\label{r:ipoCantor}
Since \(\LLN(S_{\A}) = 0\),
assumption \eqref{f:ipo} 
is equivalent to \(|D^d u|(S_{\A}) = 0\).
In particular, it is satisfied, for example,
if \(S_{\A}\) is 
$\sigma$--finite with respect to $\H^{N-1}$
(see \cite[Proposition~3.92(c)]{AFP}).
This is always the case if 
\(\A\in BV_{{\rm loc}}(\Omega, \R^N)\cap L^\infty_{{\rm loc}}(\Omega, \R^N)\) and,
notably, if $N=1$.
Another relevant situation for which \eqref{f:ipo} holds
happens when $D^c u = 0$,  
i.e.\ if 
\(u\) is a special function of bounded variation,
e.g.\  
if \(u\) is
the characteristic function of a set of finite perimeter.
Finally, since the set $\jump{\A}$ defined in \eqref{f:jump} in
$\sigma$-finite with respect to $\H^{N-1}$, we remark that assumption \eqref{f:ipo} is equivalent to
$|D^c u| (S_{\A}\setminus \jump{\A}) = 0$.
\end{remark}

\begin{example}
\label{r:Cantor}
Let us show that assumption \eqref{f:ipo} is not always satisfied.
(The following construction has been suggested by
G.E.\ Comi, personal communication.)
Let $C\subset [0,1]$ be the usual Cantor set, obtained
removing at the first step the interval $I_1^1 := (1/3, 2/3)$ from $[0,1]$,
then at the second step the intervals $I_2^1 := (1/9, 2/9)$ and
$I_2^2 := (7/9, 8/9)$ from the two remaining intervals,
and, in general, removing at the $n$-th step
$2^{n-1}$ intervals $I_n^k$, $k=1,\ldots, 2^{n-1}$, of length $3^{-n}$.
Let us consider the set
\[
E := \bigcup_{j=1}^\infty E_{2j},
\quad\text{where}\quad
E_n := \bigcup_{k=1}^{2^{n-1}} I_{n}^k\,.
\]
(In other words, $E$ is the union of the open intervals
removed at even steps.)
It is not difficult to check that $\partial E = C$.

Moreover, we claim that the following (very rough) estimates hold:
\[
\frac{1}{54} \leq
\liminf_{r \searrow 0} \frac{|B_r(x) \cap E|}{2r}
\leq
\limsup_{r \searrow 0} \frac{|B_r(x) \cap E|}{2r}
\leq
\frac{53}{54}\,,
\qquad \forall x\in C.
\]
Namely, let $x\in C$, let $r \in (0, 1/3)$,
and let $N\in\N$ be such that
$3^{-2N-1} \leq r < 3^{-2N+1}$.
Clearly, the interval $B_r(x)$ contains at least one of the intervals
of length $3^{-2N-2}$ removed at step $2N+2$, so that
\[
\frac{|B_r(x) \cap E|}{2r} \geq \frac{3^{-2N-2}}{2\cdot 3^{-2N+1}}
= \frac{1}{54}\,.
\]
A similar argument shows that 
\[
\frac{|B_r(x) \cap E|}{2r}
\leq \frac{53}{54}\,,
\]
and the claim follows.

As a consequence of the above claim,
we have that the approximate discontinuity set of $\chi_E$
coincides with $C$.

Let us consider the vector field
$\A\colon\R^2 \to \R^2$ defined by
$\A(x,y) := (0, \chi_E(x))$.
It is clear that $\A\in L^\infty(\R^2, \R^2)$,
$\Div\A = 0$
and $S_{\A} = C \times \R$.
On the other hand,
if $\psi(x)$ is the standard
Cantor--Vitali function (extended to $0$ for $x<0$ and to $1$ for $x>1$),
then the function $u(x,y) := \psi(x)$
belongs to $\BVLloc[\R^2]$
and 
$|D^c u| (S_{\A} \cap ((0,1)\times(a,b)))
= b-a$,
for every $(a,b)\subset\R$.
\end{example}

\begin{remark}
[\(BV\) vector fields]	
If \(\A \in BV_{\rm loc}(\Omega, \R^N) \cap L^\infty_{{\rm loc}}(\Omega, \R^N)\),
then clearly \(\A \in \DMloc\) and
\[
\Trace{\A}{J_u} = \A^\pm_{J_u}\cdot \nu_u\,, 
\qquad
\text{\(\hh\)-a.e.\ in}\ J_u,
\]
where \(\A^\pm_{J_u}\) are the traces of \(\A\) on \(J_u\)
(see \cite[Theorem~3.77]{AFP}).
Hence, the jump part of \((\A, Du)\) can be written as
\[
(\A, Du)^j = \frac{\A^+ + \A^-}{2}\, \cdot D^j u.
\]
\end{remark}

\begin{proof}[Proof of Theorem~\ref{t:pairing}]
Let \(u_\eps := \rho_\eps \ast u\).
It has been proved in \cite[Theorem 3.2]{ChenFrid}
that
\[
\pscal{(\A, Du)}{\varphi} =
\lim_{\eps \to 0} \pscal{(\A, Du_\eps)}{\varphi}
=
\lim_{\eps\to 0} \int_{\Omega} \varphi \, \A\cdot \nabla u_\eps\, dx,
\qquad
\forall\varphi\in C^\infty_0(\Omega)
\]
and that (i) holds.
We remark that, if 
\(K \Subset U\subset\overline{U}\Subset\Omega\) with \(U\) open, then
\[
|(\A, Du)| (K) \leq \|\A\|_{L^\infty(U)}\, |Du| (U),
\]
hence, in particular
\[
	|(\A, Du)| (E) \leq \|\A\|_{L^\infty(U)}\, |Du| (E)
\qquad
\text{for every Borel set}\ E\subset U.
\]

It remains to prove (ii) and (iii).
In order to simplify the notation, let us denote
\(\mu := (\A, Du)\).

\smallskip
\textsl{Proof of (ii)}.
Since \((\A, Du) \ll |Du|\), it is clear that
\((\A, Du)^j\) is supported in \(J_u\).
From \eqref{f:anz} and \eqref{f:truA} we have that
\[
\begin{split}
(\A, Du)^j = {} & (\A, Du) \res J_u 
= \Div(u\A)\res J_u - u^* \Div\A \res J_u
\\ 
= {} &
\left[
u^+ \Trp{\A}{J_u} - u^- \Trm{\A}{J_u}
\right] \hh\res J_u
\\ & -
\frac{u^+ + u^-}{2}
\left[
\Trp{\A}{J_u} - \Trm{\A}{J_u}
\right] \hh\res J_u
\\ = {} &
\frac{\Trp{\A}{J_u}+\Trm{\A}{J_u}}{2}
\, (u^+-u^-) \, \hh \res J_u,
\end{split}
\]
and the proof is complete.

\smallskip
\textsl{Proof of (iii)}.
Let us consider the polar decomposition 
\(D u = \polar_u \, |Du|\) of \(Du\).
By assumption \eqref{f:ipo}, the approximate limit \(\widetilde{\A}\)
of \(\A\) exists \(|D^d u|\)-a.e.\ in \(\Omega\).
Hence, the equality in (iii) is equivalent to
\[
\frac{d\mu}{d|D^du|}(x) = \frac{d\mu^{d}}{d|D^du|}(x)= \widetilde{\A}(x) \cdot \polar_u(x)
\qquad
\text{for $|D^d u|$-a.e.\ $x\in \Omega$}.
\]

Let us choose $x\in \Omega$ such that
\begin{itemize}
	\item[(a)]
	\(x\) belongs to the support of \(D^d u\), that is
	\(|D^d u|(B_r(x)) > 0\) for every \(r >0\);
	\item[(b)] 
	there exists the limit
	\(\displaystyle
	\lim_{r\to 0}\frac{\mu^d(B_r(x))}{|D^d u|(B_r(x))};
	\)
	\item[(c)]
	\(\displaystyle
	\lim_{r\to 0}\frac{|D^j u|(B_r(x))}{|D u|(B_r(x))}=0;
	\)
	\item[(d)]
	\(\displaystyle
	\lim_{r\to 0}\frac1{|D^du|(B_r(x))}\int_{B_r(x)}\left|
	\widetilde{\A}(y)\cdot\theta_u(y) - \widetilde{\A}(x)\cdot\polar_u(x)
	\right|\,d|D^d u|(y)=0.
	\)
\end{itemize}
We remark that these conditions are satisfied
for $|D^d u|$-a.e.\ $x\in\Omega$. 

Let $r>0$ be such that
\begin{equation}
\label{gtgt8}
|D^ju|\left(\partial B_r(x)\right)=0.
\end{equation}
Observe that
\(\nabla u_ \varepsilon = \rho_ \varepsilon \ast Du=
\rho_ \varepsilon \ast D^d u+\rho_ \varepsilon \ast D^ju\).
Hence for every $\phi\in C_0(\R^N)$ with support in \(B_r(x)\) it holds
\begin{equation}\label{f:diseq}
\begin{split}
&\Bigg|\frac1{|D^d u|(B_r(x))}\int_{B_r(x)}\phi(y)
\A(y)\cdot\rho_\eps\ast Du (y) \, dy
\\ & \quad 
-
\frac1{|D^d u|(B_r(x))}\int_{B_r(x)}\phi(y)
\widetilde{\A}(x)\cdot\polar_u(x)   \,d|D^d u|(y)\Bigg|\\
& \leq
\Bigg|\frac1{|D^d u|(B_r(x))}\int_{B_r(x)}\phi(y)
\A(y)\cdot \rho_\eps \ast D^d u(y)\, dy
\\ & \quad -
\frac1{|D^d u|(B_r(x))}\int_{B_r(x)}\phi(y)
\widetilde{\A}(x) \cdot \polar_u(x) \,d|D^d u|(y)\Bigg|
\\ & \quad 
+ \frac1{|D^d u|(B_r(x))}\|\phi\|_\infty \|\A\|_{L^\infty(B_r(x))}
\int_{B_r(x)}\rho_ \varepsilon \ast |D^ju|\,dy,
\end{split}
\end{equation}
where in the last inequality we use that 
$\left|\rho_ \varepsilon \ast D^ju\right|\leq \rho_ \varepsilon \ast |D^ju|$.

We note that by \eqref{gtgt8}
\[
\lim_{\varepsilon\to 0}\int_{B_r(x)}\rho_ \varepsilon \ast |D^j u|\,dy= 
|D^j u| (B_r(x)).
\]
Furthermore,
\[
\int_{B_r(x)}\phi(y)
\A(y)\cdot \rho_\eps \ast D^d u(y)\, dy 
= \int_{B_r(x)}
[\rho_\eps \ast (\phi \A)](y)\cdot\theta_u(y)\, d|D^du|(y).
\]
Hence by taking the limit as $\varepsilon\to 0$ in \eqref{f:diseq} we obtain
\[
\begin{split}
&\Bigg|\frac1{|D^d u|(B_r(x))}\int_{B_r(x)}\phi(y)\ d\mu(y)\\&-
\frac1{|D^d u|(B_r(x))}\int_{B_r(x)}\phi(y)
\widetilde{\A}(x)\cdot\polar_u(x)
\,d|D^d u|(y)\Bigg|\\
&\leq 
\frac1{|D^d u|(B_r(x))}\int_{B_r(x)}\phi(y)
\left|
\widetilde{\A}(y)\cdot\polar_u(y) - \widetilde{\A}(x)\cdot\polar_u(x)
\,d|D^d u|(y)\right|\\
&\quad +
\frac1{|D^d u|(B_r(x))}\|\phi\|_\infty \|\A\|_{L^\infty(B_r(x))}\,|D^j u|(B_r(x)).
\end{split}
\]
When $\phi(y)\to 1$ in $B_r(x)$, with \(0\leq \phi \leq 1\), we get
\[
\begin{split}
&\left|\frac{\mu(B_r(x))}{|D^d u|(B_r(x))}- 
\widetilde{\A}(x)\cdot\polar_u(x)
\right|
\\
& \leq
\frac1{|D^d u|(B_r(x))}\int_{B_r(x)}
\left|
\widetilde{\A}(y)\cdot\polar_u(y) -
\widetilde{\A}(x)\cdot\polar_u(x)
\right|\, d|D^d u|(y)\\
& \quad +
\frac1{|D^d u|(B_r(x))} \|\A\|_{L^\infty(B_r(x))} |D^j u|(B_r(x)).
\end{split}
\]
The conclusion is achieved now by taking $r\to 0$ and by using (c) and (d).
\end{proof}

\begin{example}[Computation of weak normal traces]
For illustrative purposes,
in this example we shall explicitly compute the
weak normal traces of a vector field $\A$ and
of the product $u\A$.

Let \(\A\colon\R^2\to\R^2\) be the vector field defined by
\(\A(x_1,x_2) = (1,0)\) if \(x_1 > 0\),
\(\A(x_1,x_2) = (-1,0)\) if \(x_1 < 0\).
Clearly \(\A\in\DM\) and \(\Div \A = 2 \H^1\res S\), where \(S := \{0\}\times\R\).

Let \(E := (0,1)\times (0,1)\) and let \(u := \chi_E \in BV(\R^2)\).
Let us choose on \(J_u = \partial E\) the 
orientation given by the
interior unit normal \(\nu\)
to \(E\), so that \(u^+= 1\) and \(u^- = 0\) on \(\partial E\).

Let us compute the normal traces 
\(\alpha^\pm := \Trace{\A}{J_u}\) of \(\A\) on \(J_u\),
using the construction described in Section~\ref{distrtraces}.
Let \(\partial E = J_u = S_1 \cup S_2 \cup S_3 \cup S_4\),
where
\[
S_1 = \{0\}\times [0,1],\
S_2 = [0,1]\times \{1\},\
S_3 = \{1\}\times [0,1],\
S_4 = [0,1]\times \{0\}.
\]

Let us start with the computation of the normal traces on \(S_1\).
We can construct two open domains \(\Omega\) and \(\Omega'\) of class \(C^1\),
such that \(\Omega \subset \{x_1 < 0\}\), \(\Omega'\subset\{x_1 > 0\}\),
and \(S_1 \subset \partial\Omega \cap \partial\Omega'\).
Indeed, with this choice we have
\[
\nu = \nu_\Omega = (1,0) = - \nu_{\Omega'} \qquad
\text{on}\ S_1.
\]
(Recall that \(\nu_\Omega\) is by definition the outward normal vector to \(\Omega\).)
We thus have
\[
\alpha^- := \Trace[]{\A}{\partial\Omega} = -1,
\quad
\alpha^+ := -\Trace[]{\A}{\partial\Omega'} = 1,
\qquad \text{on}\ S_1.
\]
With similar constructions we get \(\alpha^\pm = -1\) on \(S_3\)
and \(\alpha^\pm = 0\) on \(S_2 \cup S_4\), so that
\[
\alpha^* := \frac{\alpha^+ + \alpha^-}{2} =
\begin{cases}
-1, &\text{on}\ S_3,\\
0, & \text{on}\ S_1\cup S_2 \cup S_4.
\end{cases}
\]
We can now check the validity of the relation
\[
\Div(u\A) = u^*\, \Div \A + (\A, Du),
\]
where \((\A, Du) = (u^+ - u^-) \alpha^*\, \H^1\res {J_u}\)
(in this case the measure \((\A, Du)\) does not have a diffuse part).
Indeed, we have
\[
\Div(u\A) = \H^1\res {S_1} - \H^1\res {S_3},\quad
u^*\, \Div\A = \H_1\res {S_1}, \quad
(u^+ - u^-) \alpha^*\, \H^1\res {J_u} = - \H^1\res S_3.
\]
By the way, observe that \(u\A = u \boldsymbol{C}\), where \(\boldsymbol{C}\)
is the constant vector field \(\boldsymbol{C} \equiv (1,0)\) on \(\R^2\).
In this case the normal traces \(\gamma^\pm\) of \(\boldsymbol{C}\) on \(J_u\)
are \(\gamma^\pm = 1\) on \(S_1\),
\(\gamma^\pm = -1\) on \(S_3\),
\(\gamma^\pm = 0\) on \(S_2\cup S_4\),
hence
\[
u^* \Div \boldsymbol{C} = 0,
\quad
(u^+ - u^-)\, \gamma^*\, \H^1\res J_u
= \H^1\res {S_1} - \H^1\res {S_3}\,.
\]
\end{example}

\section{Chain rule, coarea and Leibniz formulas}\label{s:form}

In this section
we will prove a very general coarea formula
(see Theorem~\ref{t:coarea}).
As a consequence,
we will prove
the chain rule formula for the pairing $(\A, Dh(u))$
(see Proposition~\ref{p:comp}),
and the Leibniz formula
for $(\A, D(uv))$ and $(v\A, Du)$
(see Propositions~\ref{p:uA} and~\ref{p:leibniz}).

Finally, we will prove an approximation result 
by regular vector fields
(see Proposition~\ref{p:approx})
and a semicontinuity result
(see Proposition~\ref{approx2}).

\medskip
Since the measure $(\A, Du)$ is absolutely continuous with respect to $|Du|$,
then
\begin{equation}\label{theta}
(\A, Du) = \theta(\A,Du,x) \, {|Du|},
\end{equation}
where \(\theta(\A, Du, \cdot)\) denotes
the Radon--Nikod\'ym derivative of $(\A, Du)$ with respect to $|Du|$.

Let $Du = \polar_u |Du|$ be the polar decomposition of $Du$.
From Theorem~\ref{t:pairing}, if \(|D^c u|(S_{\A}) = 0\) it holds
\begin{equation}
\label{f:theta}
\theta(\A, Du, x) =
\begin{cases}
\pscal{\widetilde{\A}(x)}{\polar_u(x)},
& \text{for \(|D^d u|\)-a.e.}\ x\in\Omega,\\
\alpha^*(x) \sign(u^+(x)-u^-(x)),
& \text{for \(\H^{N-1}\)-a.e.}\ x\in J_u,
\end{cases}
\end{equation}
where $\alpha^* := [\Trp{\A}{J_u}+\Trm{\A}{J_u}]/2$.
 
\begin{remark}
If $\Div \A \in L^1(\Omega)$
and $u\in BV(\Omega)\cap L^\infty(\Omega)$, then
$\Trp{\A}{J_u}=\Trm{\A}{J_u}$ $\H^{N-1}$-a.e.\ in $J_u$.
Moreover,
Anzellotti has proved in \cite[Theorem~3.6]{Anz2}
that
\[
\theta(\A, Du, x) = q_{\A} (x, \polar_u(x))
\qquad
\text{for $|Du|$-a.e.}\ x\in\Omega,
\] 
where, for every $\zeta\in S^{N-1}$,
\[
q_{\A}(x, \zeta)
:= \lim_{\rho\downarrow 0}
\lim_{r\downarrow 0}
\frac{1}{\LLN(C_{r,\rho}(x, \zeta))}
\int_{C_{r,\rho}(x, \zeta)} \A(y) \cdot \zeta\, dy
\]
with
\[
C_{r,\rho}(x, \zeta) :=
\left\{
y\in\R^N:\ 
|(y-x)\cdot\zeta| < r,\
|(y-x) - [(y-x)\cdot\zeta]\zeta| < \rho
\right\}
\]
(the existence of the limit in the definition of $q_{\A}(x, \polar_u(x))$
for $|Du|$-a.e.\ $x\in\Omega$ is part of the statement).
By using \eqref{f:theta} in this framework, we can conclude that if $\Div \A \in L^1(\Omega)$ and 
\(|D^c u|(S_{\A}) = 0\), then we have
\[
\pscal{\widetilde{\A}(x)}{\polar_u(x)} = q_{\A}(x, \polar_u(x))
\qquad
\text{for $|D^d u|$-a.e.}\ x\in\Omega.
\]
Finally, we remark that,
when $\A$ is a $W^{1,1}(\Omega;\R^N)$ vector field, then $\Div \A \in L^1(\Omega)$
and \(|D^c u|(S_{\A}) = 0\).
\end{remark}

\begin{theorem}[Coarea formula]\label{t:coarea}
	Let $\A\in\DMloc[\Omega]$, let 
	$u\in BV_{{\rm loc}}(\Omega)$
	and assume that
	$u^*\in L^1_{\rm{loc}}(\R^N,\Div \A)$.
	Then
	\begin{equation}\label{f:split}
	\pscal{(\A, Du)}{\varphi} = \int_{\R} \pscal{(\A, D\chiut)}{\varphi}\, dt,
	\qquad\forall \varphi\in C_c(\Omega)
	\end{equation}
	and, for any Borel set \(B\subset\Omega\),
	\begin{equation}\label{f:split2}
	(\A, Du)(B) = \int_{\R} (\A, D\chiut)(B)\, dt.
	\end{equation}
	Furthermore, for \(\mathcal{L}^1\)-a.e.\ \(t\in\R\),
	\begin{equation}
	\label{f:thetaut}
	\theta(\A, Du, x) =
	\theta(\A, D\chiut, x)
	\qquad
	\text{for \(|D\chiut|\)-a.e.}\ x\in\Omega.
	\end{equation}
\end{theorem}

\begin{remark}\label{lincei33}
	Formulas \eqref{f:split} and \eqref{f:split2} have been proved by Anzellotti
	(see \cite[Proposition~2.7]{Anz}) for
	\(u\in BV(\Omega)\) and
	\(\A\in L^\infty(\Omega, \R^N)\) with \(\Div \A \in L^N(\Omega)\).
	Moreover they have been proved in \cite[Propositions 2.4 and 2.5]{LaSe}
	when \(D^j u = 0\).
\end{remark}

\begin{proof}
	Let us first consider the case \(u\in L^\infty(\Omega)\).
	By possibly replacing \(u\) with \(u +\|u\|_\infty\),
	it is not restrictive to assume that \(u\geq 0\)

	Let us fix a test function \(\varphi\in C^\infty_c(\Omega)\).
	From the definition \eqref{f:pairing} of the pairing, we have that
	\begin{equation}
	\label{f:I1I2}
	\begin{split}
	\int_{\R} \pscal{(\A, D\chiut)}{\varphi}\, dt  = {} &
	-\int_0^{+\infty} \left(
	\int_{\Omega} \chiut[*] \varphi\, d\Div\A
	\right)\, dt 
	\\ & -
	\int_0^{+\infty} \left(
	\int_{\Omega} \chiut \A\cdot\nabla\varphi\, dx
	\right)\, dt
	=: -I_1 - I_2. 
	\end{split}
	\end{equation}
	The integral \(I_2\) can be immediately computed as
	\begin{equation}\label{f:I2}
	I_2 = \int_{\Omega} u\, \A\cdot\nabla\varphi\, dx.
	\end{equation}
	
	The first integral \(I_1\) requires more care.
	From \cite[Lemma~2.2]{DCFV2} we have that,
	for \(\mathcal{L}^1\)-a.e.\ \(t\in\R\),
	there exists a Borel set \(N_t\subset\Omega\), with \(\hh(N_t) = 0\),
	such that
	\[
	\forall x\in \Omega\setminus N_t:
	\qquad
	\chiut[*](x) =
	\begin{cases}
	1, &\text{if}\ u^-(x) > t,\\
	0, & \text{if}\ u^+(x) < t,\\
	1/2, & \text{if}\ u^-(x) \leq t \leq u^+(x)\,.
	\end{cases}
	\]
	Since \(|\Div\A| \ll \hh\), we deduce that,
	for \(\mathcal{L}^1\)-a.e.\ \(t\in\R\),
	\begin{equation}\label{f:chis}
	\chiut[*](x) = \frac{\chi_{\{u^- > t\}}(x) + \chi_{\{u^+ > t\}}(x)}{2}\,,
	\qquad
	\text{for \(|\Div\A|\)-a.e.}\ x\in\Omega.
	\end{equation}
	From \eqref{f:chis}, we can rewrite \(I_1\) in the following way:
	\begin{equation}
	\label{f:I1}
	\begin{split}
	I_1 & = 
	\int_0^{+\infty} \int_{\Omega}\left(
	\frac{\chi_{\{u^- > t\}} + \chi_{\{u^+ > t\}}}{2}\, \varphi\, d\Div\A
	\right)\, dt
	\\ & =
	\int_\Omega \frac{u^- + u^+}{2} \, \varphi\, d\Div\A
	= \int_\Omega u^* \, \varphi\, d\Div\A\,.
	\end{split}
	\end{equation}
	Hence, from \eqref{f:I1I2}, \eqref{f:I2}, \eqref{f:I1}  and
	the definition \eqref{f:pairing} of \((\A, Du)\),
	we conclude that \eqref{f:split} holds
	for every test function \(\varphi\in C^\infty_c(\Omega)\).
	On the other hand, since both sides in \eqref{f:split} are
	measures in \(\Omega\), they coincide not only
	as distributions, but also as measures.
	Hence \eqref{f:split} and \eqref{f:split2} follow.
	
	Since, for \(\mathcal{L}^1\)-a.e.\ \(t\in\R\), it holds
	\[
	\frac{d Du}{d|Du|} = \frac{d D\chiut}{d|D\chiut|}
	\qquad
	\text{\(|D\chiut|\)-a.e.\ in}\ \Omega,
	\]
	we conclude that \eqref{f:thetaut} follows.
	
	Finally, the general case $u^*\in L^1_{\rm{loc}}(\R^N,\Div \A)$ follows
	using the previous step on the truncated
	functions \(u_k := T_k(u)\),
where, given
$k>0$, $T_k$ is defined by
\begin{equation}\label{trun}
T_k(s) := \max\{\min\{s, k\}, -k\}
\,,
\qquad s\in\R.
\end{equation}
Since $T_k$ is a Lipschitz continuous function, we get that
\[
u_k\in BV_{{\rm loc}}(\Omega) \cap L^\infty(\Omega),
\quad 
u_k^\pm = T_k(u^\pm), 
\quad 
|D u_k|\leq |Du|\ \text{in the sense of measures}.
\] 
Then $|u_k^\pm|\leq |u^\pm|$ and 
$|u_k^*|\leq |u^*|$, which implies that $u_k^*\in L^1_{\rm{loc}}(\Omega,\Div \A)$. 
\end{proof}

\begin{remark}[Representation of \(\theta(\A, Du, x)\)]
Let $\A\in\DMloc[\Omega]$ and let $u\in \BVLloc[\Omega]$. 
If \(E\Subset\Omega\) is a set of finite perimeter, then
\(|D\chi_E | = \hh\res\partial^*E\) hence,
by Theorem~\ref{t:pairing}, we have that
\[
(\A, D\chi_E) = \frac{\Trp{\A}{\partial^*E} + \Trm{\A}{\partial^*E}}{2}\,
|D\chi_E|,
\]
that is
\[
\theta(\A, D\chi_E, x) = \frac{\Trp{\A}{\partial^*E} + \Trm{\A}{\partial^*E}}{2}
\qquad
\text{for \(\hh\)-a.e.}\ x\in\partial^*E.
\]
Since, for \(\mathcal{L}^1\)-a.e.\ \(t\in\R\), the set
\(E_{u,t} := \{u>t\}\) is of finite perimeter, then from \eqref{f:thetaut} we deduce that,
for these values of \(t\),
\[
\theta(\A, Du, x) = \frac{\Trp{\A}{\partial^* E_{u,t}} + \Trm{\A}{\partial^*E_{u,t}}}{2}
\qquad
\text{for \(\hh\)-a.e.}\ x\in\partial^*E_{u,t}.
\]
\end{remark} 
 
\begin{proposition}[Chain Rule]\label{p:comp} 
	Let $\A\in\DMloc[\Omega]$ and let $u\in \BVLloc[\Omega]$. 
	Let $h:\R\to\R$ be a locally Lipschitz function.
	Then the following properties hold:
	\begin{itemize}
		\item[(i)]
		\((\A, Dh(u))^a = h'(\ut)\, \A \cdot \nabla u \, \LLN\) and,
		if \(|D^c u|(S_{\A}) = 0\), then
		\((\A, Dh(u))^d = h'(\ut)\, (\A, Du)^d\);
		\item[(ii)]
		\(\displaystyle
		(\A, Dh(u))^j = \frac{h(u^+) - h(u^-)}{u^+-u^-} (\A, Du)^j;\)
		\item[(iii)]
		if \(h\) is non-decreasing, then
		\begin{equation}\label{thetacomp}
		\theta(\A,Dh(u),x)=\theta(\A,Du,x),
		\qquad \text{for \(|Dh(u)|\)-a.e.}\ x\in\Omega.
		\end{equation} 
	\end{itemize}
\end{proposition}

\begin{remark} 
	Formula \eqref{thetacomp} has been proved by Anzellotti
	(see \cite[Proposition 2.8]{Anz}) for
	\(h\in C^1\), \(u\in BV(\Omega)\) and
	\(\A\in L^\infty(\Omega, \R^N)\) with \(\Div \A \in L^N(\Omega)\).
	Moreover, it has been proved in \cite[Proposition 2.7]{LaSe}
	when \(D^j u = 0\).
\end{remark}

\begin{remark}
The same characterization of \((\A, Dh(u))\) holds true if \(h\colon I\to\R\) is a locally Lipschitz
function in a interval \(I\), provided that \(u(\Omega)\subset I\) and
\(h\circ u \in BV_{{\rm loc}}(\Omega)\).
\end{remark}

\begin{proof} 
	From the Chain Rule Formula (see \cite[Theorem 3.99]{AFP}), we have that
\[
D^d h(u) = h'(\ut) D^d u,
\qquad
D^j h(u) = (h(u^+) - h(u^-))\, \nu \, \hh \res J_u.
\]	
On the other hand, \((h(u))^\pm = h(u^\pm)\),
hence (i) and (ii) follow
	from Theorem~\ref{t:pairing}. 

It remains to prove (iii).
If $h$ is a strictly increasing function,
then formula \eqref{thetacomp} can be proved as in \cite[Proposition 2.8]{Anz} by using the consequence \eqref{f:thetaut} of the coarea formula. 
The case of $h$ non-decreasing function can now be handled as in
\cite[Proposition~2.7]{LaSe}.
\end{proof}

\taglio{
\noindent
\hrulefill

The following result is
an immediate consequence of Theorem~\ref{t:pairing} and
Proposition~\ref{p:tracesb}.

\begin{corollary}\label{c:uA}
	Let \(\A\in\DMloc\) and \(u, v\in\BVLloc\).
	Then the measure \((v\A, Du)\) admits the following decomposition:
	\begin{itemize}
		\item[(i)]
		absolutely continuous part: 
		\((v \A, Du)^a = v \A \cdot \nabla u\, \LLN\);

		\item[(ii)]
		jump part:
		\(\displaystyle
		(v \A, Du)^j = 
		\frac{v^+\Trp{\A}{J_u}+v^-\Trm{\A}{J_u}}{2}
		\, (u^+-u^-) \, \hh \res J_u
		\),
		where we understand \(v^\pm \equiv \widetilde{v}\) 
		in \(\Omega\setminus S_v\);
		
		\item[(iii)]
		diffuse part:
		if, in addition, \eqref{f:ipo} holds, 
		then
		\((v \A, Du)^d = \widetilde{v}\, \widetilde{\A} \cdot D^d u\).
	\end{itemize}
\end{corollary}

\begin{remark}
	Observe that, in general,
	\[
	(v \A, Du) \neq v^* (\A, Du),
	\]
	because the jump part of the two measures can differ on points
	of \(J_u\).
	This can be clearly seen when \(u=v\) in Lemma~\ref{l:uA} below
	(see also the case of \(u = v = \chi_E\) in \cite[Remark 3.4]{ComiPayne}).
\end{remark}

The case \(u=v\) in Corollary~\ref{c:uA} can be completely characterized even
without the additional assumption \eqref{f:ipo}.

\noindent
\hrulefill
} 

The aim of the following results is the characterization of the pairing
\((v\A, Du)\).
We first present a preliminary result in the case \(u = v\) in Lemma~\ref{l:uA}.
The general case will follow in Proposition~\ref{p:uA}.
The same results, under the assumption \(D^j u = D^j v = 0\),
have been proven in \cite[Proposition~2.3]{MaSe}.

\begin{lemma}\label{l:uA}
	Let \(\A\in\DMloc\) and \(u\in\BVLloc\). Then
	\begin{equation}\label{f:uADu}
	(u\A, Du) = u^* (\A, Du) + \frac{(u^+ - u^-)^2}{4} \Div\A\res J_u,
	\end{equation}
	that is
	\begin{gather}
	(u\A, Du)^d = u^* (\A, Du)^d,\label{f:uADud}\\
	(u\A, Du)^j = \frac{\alpha^+ u^+ + \alpha^- u^-}{2}\, (u^+ - u^-)\, \hh\res J_u,\label{f:uADuj}
	\end{gather}
	where \(\alpha^\pm := \Trace{\A}{J_u}\).
	In particular, if \(D^j u = 0\) then \((u\A, Du) = u^* (\A, Du)\).
\end{lemma}

\begin{proof}
Since the statement is local in nature,
it is not restrictive to assume that \(u\in L^\infty(\Omega)\).

Let us first assume that \(u > 0\).
Since \(D(u^2) = 2 u^* Du\),
from Proposition~\ref{p:comp}(iii) we have that
\[
\theta(\A, D(u^2), x) = \theta(\A, Du, x)
\qquad
\text{for \(|Du|\)-a.e.}\ x\in\Omega,
\]
hence
\[
(\A, D(u^2)) = \theta(\A, D(u^2), x) |D(u^2)|
= \theta(\A, Du, x) 2 u^* |Du| = 2 u^* (\A, Du).
\]
Starting from the relation
\[
\Div(u^2\A) = (u^2)^* \Div\A + (\A, D(u^2))
= (u^2)^* \Div\A + 2 u^* (\A, Du)
\] 	
we get
\[
\begin{split}
2 u^* (\A, Du) & = \Div(u^2\A) - (u^2)^* \Div\A
= u^* \Div(u\A) + (u\A, Du) -  (u^2)^* \Div\A
\\ & = [(u^*)^2 - (u^2)^*] \Div\A + u^* (\A, Du) + (u\A, Du),
\end{split}
\]
that is
\[
(u\A, Du) = u^* (\A, Du) - [(u^*)^2 - (u^2)^*] \Div\A.
\]
Hence \eqref{f:uADu} follows after observing
that \((u^*)^2 - (u^2)^* = 0\) in \(\Omega\setminus S_u\)
and \((u^*)^2 - (u^2)^* = - (u^+ - u^-)^2 / 4\) on \(J_u\).
The relations \eqref{f:uADud} and \eqref{f:uADuj} now follow
from Theorem~\ref{t:pairing}(ii).

\smallskip
The general case of \(u\in L^\infty(\Omega)\) can be obtained
from the previous case, considering the function \(v := u + c\),
which is positive if \(c > \|u\|_\infty\).
Namely, \eqref{f:uADu} easily follows observing that
\[
(v\A, Dv) = (u\A, Du) + c (\A, Du),
\quad
v^* = u^* + c,
\quad
J_v = J_u,
\quad
v^+ - v^+ = u^+ - u^-.
\qedhere
\]
\end{proof}

\begin{proposition}\label{p:uA}
	Let \(\A\in\DMloc\) and \(u, v\in\BVLloc\).
	Then 
\begin{equation}\label{f:vADu}
(v\A, Du) =
v^* (\A, Du) + \frac{(u^+ - u^-)(v^+ - v^-)}{4}\, \Div\A \res (J_u \cap J_v),
\end{equation}
that is
\begin{gather}
(v\A, Du)^d = v^* (\A, Du)^d,\label{f:vADud}\\
(v\A, Du)^j =  \frac{\alpha^+ v^+ + \alpha^- v^-}{2}\,(u^+ - u^-) \, \hh\res J_u,\label{f:vADuj}
\end{gather}
where \(\alpha^\pm := \Trace{\A}{J_u}\).
\end{proposition}

\begin{proof}
From Lemma~\ref{l:uA} we have that
\begin{equation}\label{f:LH}
\begin{split}
((u+v)\A, D(u+v)) = {} & (u+v)^* (\A, D(u+v)) 
\\ & + 
\frac{(u^+ + v^+ - u^- - v^-)^2}{4}\, \Div\A \res (J_u\cup J_v).
\end{split}
\end{equation}
Let us compute the two sides of this equality.
We have that
\[
\begin{split}
LHS  = {} & 
(u\A, Du) + (v\A, Dv) + (v\A, Du) + (u\A, Dv)
\\ = {} &
u^* (\A, Du) + \frac{(u^+-u^-)^2}{4} \Div\A\res J_u
+ v^* (\A, Dv) + \frac{(v^+ - v^-)^2}{4}\, \Div\A \res J_v
\\ & + (v\A, Du) + (u\A, Dv)
\end{split}
\]
On the other hand, the right--hand side of \eqref{f:LH} is computed as
\[
\begin{split}
RHS = {} &
u^* (\A, Du) + u^*(\A, Dv) + v^*(\A, Du) + v^* (\A, Dv)
\\ & + \frac{(u^+-u^-)^2}{4}\, \Div\A\res J_u
+ \frac{(v^+-v^-)^2}{4}\, \Div\A\res J_v
\\ & + \frac{(u^+-u^-)(v^+-v^-)}{2}\, \Div\A\res (J_u\cap J_v).
\end{split}
\]
Hence, after some simplifications \eqref{f:LH} gives
\begin{equation}\label{f:eq1}
\begin{split}
(v\A, Du) + (u\A, Dv)  = {} &
u^*(\A, Dv)  + v^* (\A, Du)
\\ & + \frac{(u^+-u^-)(v^+-v^-)}{2}\, \Div\A\res (J_u\cap J_v).
\end{split}
\end{equation}
Since
\[
\Div(uv\A) = u^*\Div(v\A) + (v\A, Du), 
\qquad
\Div(uv\A) = v^*\Div(u\A) + (u\A, Dv), 
\]
it holds
\begin{equation}\label{f:uvA}
(v\A, Du) - (u\A, Dv) = v^* (\A, Du) - u^*(\A, Dv).
\end{equation}
Summing together \eqref{f:eq1} and \eqref{f:uvA} we get \eqref{f:vADu}.
The relations \eqref{f:vADud} and \eqref{f:vADuj} now follow
from Theorem~\ref{t:pairing}(ii).
\end{proof}

\begin{remark}
	Observe that, in general,
	\[
	(v \A, Du) \neq v^* (\A, Du),
	\]
	because the jump part of the two measures can differ on points
	of \(J_u\cap J_v\)
	(see also the case of \(u = v = \chi_E\) in \cite[Remark 3.4]{ComiPayne}).
\end{remark}

\begin{proposition}[Leibniz rule]
	\label{p:leibniz}
	Let \(\A\in\DMloc\) and \(u, v\in\BVLloc\).
	Then
	\begin{equation}\label{f:leibniz}
	(\A, D(uv)) = v^*(\A, Du) + u^*(\A, Dv).
	\end{equation}
	More precisely, the measure \((\A, D(uv))\) admits the following decomposition:
	\begin{itemize}
		\item[(i)]
		absolutely continuous part: 
		\((\A, D(uv))^a = \A \cdot \nabla (uv)\, \LLN\), with
		\(\nabla(uv) = u \nabla v + v \nabla u\);
		
		\item[(ii)]
		jump part:
		\[
		\begin{split}
		& (\A, D(uv))^j = 
		\frac{\alpha^+ + \alpha^-}{2} (u^+ v^+ - u^- v^-)  \hh \res (J_u \cup J_v).
		\end{split}
		\]
		where 
		\(\alpha^\pm := \Trace{\A}{J_u \cup J_v}\);
		
		\item[(iii)]
		diffuse part:
		if, in addition, \(|D^c(uv)|(S_{\A}) = 0\),
		then
		\((\A, D(uv))^d = \widetilde{\A} \cdot D^d (uv)\), with
		\(D^d(uv) = \ut D^d v + \widetilde{v} D^d u\).
		
	\end{itemize}
\end{proposition}

\begin{proof}
We have that
\[
\begin{split}
(\A, D(uv)) = {} &
\Div(uv\A) - (uv)^* \Div\A
= \frac{1}{2} \Div(uv\A) + \frac{1}{2} \Div(uv\A) - (uv)^* \Div\A
\\ = {} &
\frac{1}{2}\left[ u^* \Div(v\A) + (v\A, Du)\right]
+ \frac{1}{2}\left[ v^* \Div(u\A) + (u\A, Dv)\right]
- (uv)^*\Div\A
\\ = {} &
\frac{1}{2}\, u^*\left[ \Div(v\A) - v^* \Div\A\right]
+ \frac{1}{2}\, v^*\left[ \Div(u\A) - u^* \Div\A\right]
\\ & +\frac{1}{2}\, (v\A, Du) + \frac{1}{2}\, (u\A, Dv)
+\left[u^*v^* - (uv)^*\right]\Div\A
\\ = {} &
\frac{1}{2}\, u^* (\A, Dv) + \frac{1}{2}\, v^* (\A, Du)
+\frac{1}{2}\, (v\A, Du) + \frac{1}{2}\, (u\A, Dv)
\\ & +\left[u^*v^* - (uv)^*\right]\Div\A\,.
\end{split}
\]
A direct computation shows that
\[
u^* v^* - (uv)^* = -\frac{(u^+ - u^-)(v^+-v^-)}{4}
\qquad \text{\(\hh\)-a.e.\ in}\ J_u\cup J_v,
\]
whereas \(u^* v^* - (uv)^* = 0\) in \(\Omega\setminus(S_u\cup S_v)\).

Hence, using \eqref{f:vADu} on \((u\A, Dv)\) and \((v\A, Du)\),
we finally get \eqref{f:leibniz}.
\end{proof}

Using the results proved so far,
Theorem~\ref{t:pairing} can be slightly extended to the case of unbounded \(BV\) functions
as follows.

\begin{theorem}\label{t:pairing2}
	Let $\A \in \DMloc$, $u\in BV_{{\rm loc}}(\Omega)$ and assume that 
	$u^*\in L^1_{\rm{loc}}(\Omega,\Div \A)$.
	Then the pairing \((\A, Du)\), defined as a distribution by \eqref{f:pairing},
	is a Radon measure in \(\Omega\) and admits the decomposition given in Theorem~\ref{t:pairing}.
\end{theorem}

\begin{proof}
	The fact that \((\A, Du)\) is a Radon measure in \(\Omega\),
	with \(|(\A, Du)| \ll |Du|\), has been proved
	in \cite[Corollary~2.3]{DeGiSe}.
	
	Properties (i), (ii) and (iii) in Theorem~\ref{t:pairing}
	will follow with a truncation argument similar to that used
	in the proof of Proposition~2.7 in \cite{Anz}.
	
	More precisely,
	let us define the truncated functions $u_k := T_k(u)$
	where \(T_k\) is defined in \eqref{trun}.

Since $|u_k^*|\leq |u^*|$,
by the Dominated Convergence Theorem we can pass to the limit in the relation
\[
\pscal{(\A, Du_k)}{\varphi} =
-\int_\Omega u_k^* \varphi\, d\Div\A - \int_\Omega u_k \A\cdot\nabla\varphi\, dx
\]
obtaining that
\[
\pscal{(\A, Du_k)}{\varphi} \to
\pscal{(\A, Du)}{\varphi}
\qquad \forall\varphi\in C^\infty_c(\Omega).
\]
Since $|D^c u_k| \leq |D^c u|$,
from Theorem~\ref{t:pairing} 
it holds
\begin{gather}
(\A, Du_k)^d = \widetilde{\A}\cdot D^d u_k,
\qquad \text{if}\ |D^c u|(S_{\A}) = 0,
\label{f:id1}\\
(\A, Du_k)^j = \Trace[*]{\A}{J_{u_k}} (u_k^+-u_k^-) \hh\res J_{u_k}
= \Trace[*]{\A}{J_{u}} (u_k^+-u_k^-) \hh\res J_{u}.\label{f:id2}
\end{gather}

From the Chain Rule Formula (see \cite[Example 3.100]{AFP})
we have that
\[
D^d u_k \res \{|\ut| < k\} = D^du \res \{|\ut| < k\}.
\]
Since, for every \(x\in\Omega\setminus S_u\) there exists \(k > 0\) such that
\(x \in \{|\ut| < k\}\), from \eqref{f:id1} we conclude that
(i) and (ii) in Theorem~\ref{t:pairing} hold.

Concerning the jump part, observe that
if \(x\in J_u\) and
\(k > \max\{|u^+(x)|, |u^-(x)|\}\),
then \(x\in J_{u_k}\) and \(u_k^\pm(x) = T_k(u^\pm(x)) = u^\pm(x)\).
Hence from \eqref{f:id2} we can conclude that also property (iii) in Theorem~\ref{t:pairing} holds.
\end{proof}

\begin{remark}
We extract the following fact from the proof of Theorem~\ref{t:pairing2}.
Let $\A \in \DMloc$, $u\in BV_{{\rm loc}}(\Omega) \cap L^1_{\rm{loc}}(\Omega,\Div \A)$,
and let \(u_k := T_k(u)\) be the truncated functions of \(u\),
where \(T_k\) is defined in \eqref{trun}.
If we define
\[
\Omega_k := \{x\in J_u:\ |u^\pm(x)| < k\} \cup \{x\in \Omega\setminus S_u:\ |\ut(x)| < k\},
\]
then it holds
\[
(\A, Du_k) \res \Omega_k
= (\A, Du) \res \Omega_k
\qquad \forall k > 0.
\]
\end{remark}

\begin{remark}
Let $\A \in \DMloc$ and $u\in BV_{{\rm loc}}(\Omega)$.
Then \(u^*\in L^1_{\rm{loc}}(\Omega,\Div \A)\)
if at least one of the following conditions holds:
\begin{itemize}
	\item[(a)]
	\(u\in L^\infty_{{\rm loc}}\);
	\item[(b)]
	\(\Div\A\geq 0\) or \(\Div\A\leq 0\).
\end{itemize}
The first case is trivial.
For case (b) the proof follows from \cite[Remark~8.3]{PhTo}.
\end{remark}

We conclude this section with an approximation result in the spirit
of \cite[Theorem~1.2]{ChenFrid}.
This kind of approximation has been used for example in \cite{Anz} and \cite{ChenFrid}
as an essential tool in order to pass from smooth vector fields to less regular fields.
Unfortunately, in our general setting, properties (iv) and (v) below can be proved only
under the additional assumption $|D^c u| (S_{\A}) = 0$, so we cannot use
this approximation to obtain the Gauss--Green formula in Section~\ref{s:GG}.
Nevertheless, we think that Proposition~\ref{approx2} may be useful
in order to get semicontinuity results for functionals depending linearly in $\nabla u$.

\begin{proposition}[Approximation by \(C^\infty\) functions]\label{p:approx}
	Let \(\A\in\DM(\Omega)\).
	Then there exists a sequence \((\A_k)_k\) in 
	\(C^\infty(\Omega,\R^N)\cap L^\infty(\Omega, \R^N)\) 
	satisfying the following properties.
	\begin{itemize}
		\item[(i)] 
		\(\A_k \to \A\) in \(L^1(\Omega,\R^N)\) and
		\(\int_\Omega |\Div \A_k|\, dx \to |\Div \A|(\Omega)\).
		\item[(ii)] 
		\(\Div\A_k \stackrel{*}{\rightharpoonup} \Div\A\)
		in the weak${}^*$ sense of measures in \(\Omega\).
		\item[(iii)]
		For every oriented countably \(\hh\)-rectifiable set \(\Sigma\subset\Omega\) it holds
		\[
		\pscal{\Trace{\A_k}{\Sigma}}{\varphi} \to
		\pscal{\Trace[*]{\A}{\Sigma}}{\varphi}
		\qquad \forall \varphi\in C_c(\Omega),
		\]
		where \(\Trace[*]{\A}{\Sigma} := [\Trp{\A}{\Sigma} + \Trm{\A}{\Sigma}]/2\).
	\end{itemize}
	If, in addition,
	\(u\in \BVLloc\) and \(|D^c u|(S_{\A}) = 0\), then
	\begin{itemize}
		\item[(iv)] 
		\((\A_k, Du) \stackrel{*}{\rightharpoonup} (\A, Du)\)
		locally in the weak${}^*$ sense of measures in \(\Omega\);
		\item[(v)]
		\(
		\theta(\A_k, Du, x) \to \theta(\A, Du, x)\)
		for \(|Du|\)-a.e.\ \(x\in\Omega\).
	\end{itemize}
\end{proposition}

\begin{remark}
	It is not difficult to show that a similar approximation result holds also
	for \(\A\in\DMloc\) with a sequence \((\A_k)\) in
	\(C^\infty(\Omega,\R^N)\). 
\end{remark}

\begin{proof}
	(i) This part is proved in \cite[Theorem 1.2]{ChenFrid}.
	We just recall, for later use, that
	for every $k$
the vector field \(\A_k\) is of the form
\begin{equation}\label{f:Akk}
\A_k = \sum_{i=1}^\infty \rho_{\eps_i} \ast (\A \varphi_i),
\end{equation}
where \((\varphi_i)\) is a partition of unity subordinate to
a locally finite covering of \(\Omega\)
depending on $k$ and, for every $i$,
\(\eps_i \in (0, 1/k)\) 
is chosen in such a way that
\begin{equation}
\label{f:epsi}
\int_\Omega \left|
\rho_{\eps_i} \ast (\A\cdot\nabla\varphi_i) - \A\cdot \nabla \varphi_i
\right|\, dx\leq \frac{1}{k\, 2^i}
\end{equation}	
(see \cite{ChenFrid}, formula (1.8)).	

\smallskip
	(ii) From (i) we have that 
	\[
	\int_\Omega \A_k\cdot\nabla\varphi\, dx \to \int_\Omega \A\cdot\nabla\varphi\, dx
	\qquad
	\forall \varphi\in C^1_c(\Omega),
	\]
	hence (ii) follows from  \(\sup_k \int_\Omega |\Div \A_k|\, dx < +\infty\)
	and the density of \(C^0_c(\Omega)\) in \(C^1_c(\Omega)\)
	in the norm of \(L^\infty(\Omega)\).

\smallskip	
	(iii)
Before proving (iii), we need to prove the following claim:
if 
\(E\Subset\Omega\) is a set of finite perimeter,
then
\begin{equation}
\label{f:ivp}
\lim_{k\to +\infty}\int_{\Omega} \chi_{E} \, \varphi \, \Div\A_k\, dx
= \int_{\Omega} \chi^*_{E} \,\varphi \, d \Div\A,
\qquad \forall \varphi\in C^\infty_c(\Omega).
\end{equation}
Namely, from the definition \eqref{f:Akk} of $\A_k$ 
and the identity $\sum_i \nabla\varphi_i = 0$
we have that
\[
\Div \A_k = \sum_i \rho_{\eps_i}\ast (\varphi_i \Div\A)
+ \sum_i \left[
\rho_{\eps_i} \ast (\A \cdot \nabla\varphi_i) - \A \cdot \nabla\varphi_i
\right].
\]
From \eqref{f:epsi} we have that
\[
\left|
\sum_i \int_\Omega \chi_{E} \varphi \left[\rho_{\eps_i} \ast (\A \cdot \nabla\varphi_i) - \A \cdot \nabla\varphi_i\right]\, dx
\right| < \frac{1}{k}\, \|\varphi\|_\infty\,,
\]
hence, to prove \eqref{f:ivp}, it is enough to show that
\begin{equation}
\label{f:limkapp}
\lim_{k\to +\infty}
\sum_i \int_\Omega \chi_{E} \varphi \, \rho_{\eps_i} \ast (\varphi_i \Div \A)
=  
\int_{\Omega} \chi^*_{E} \,\varphi \, d \Div\A.
\end{equation}
On the other hand,
\[
\sum_i \int_\Omega \chi_{E} \varphi \, \rho_{\eps_i} \ast (\varphi_i \Div \A)
=
\sum_i \int_\Omega \rho_{\eps_i}\ast ( \chi_{E} \varphi) \, \varphi_i \, d \Div \A\,,
\]
hence \eqref{f:limkapp} follows observing that,
$\H^{N-1}$--a.e.\ in $\Omega$,
\[
\chi^*_{E}\varphi - \sum_i \varphi_i \rho_{\eps_i}\ast(\chi_{E} \varphi)
= \sum_i \varphi_i \left[
\chi^*_{E}\varphi -
 \rho_{\eps_i}\ast(\chi_{E} \varphi)
\right]
\to 0.
\]

\smallskip	
Let us now prove (iii).
	Let \(\Omega'\Subset \Omega\) be a set of class \(C^1\).
	By the definition \eqref{f:disttr},
	by (i), (ii) and \eqref{f:ivp},
	for every \(\varphi\in C^\infty_c(\Omega)\) we have that
	\[
	\begin{split}
	\pscal{\Trace[]{\A_k}{\partial\Omega'}}{\varphi}
	& = \int_{\Omega'} \A_k \cdot \nabla\varphi\, dx
	+ \int_{\Omega'} \varphi \, \Div\A_k\, dx
	\\ & = 
	\int_{\Omega} \chi_{\Omega'} \, \A_k \cdot \nabla\varphi\, dx
	+ \int_{\Omega} \chi_{\Omega'} \, \varphi \, \Div\A_k\, dx
	\\ & \to 
	\int_{\Omega} \chi_{\Omega'} \,\A \cdot \nabla\varphi\, dx
	+ \int_{\Omega} \chi^*_{\Omega'} \,\varphi \, d \Div\A
	\\ & =
	\int_{\Omega'} \A \cdot \nabla\varphi\, dx
	+ \int_{\Omega'} \varphi \, d\Div\A
	+ \frac{1}{2} \int_{\partial\Omega'} \varphi\, d\Div\A.
	\end{split}
	\]
	Hence, using the notations of Section~\ref{distrtraces}, 
	by \eqref{mmm} on the set \(N_i \subset \partial\Omega_i \cap \partial\Omega_i'\)
	it holds
	\[
	\Trm{\A_k}{\Sigma} = \Trace[]{\A_k}{\partial\Omega_i}
	\to \Trm{\A}{\Sigma}
	+\frac{1}{2}\left[
	\Trp{\A}{\Sigma} - \Trm{\A}{\Sigma}
	\right]
	=\Trace[*]{\A}{\Sigma},
	\]
	where the convergence is in the weak${}^*$ sense of \(L^\infty\).
	A similar computation holds for \(\Trp{\A_k}{\Sigma}\).
	
	(iv)
	From the very definition \eqref{f:Akk} of $\A_k$, we have that
	\begin{equation}\label{f:Akd}
	\A_k(x) \to \widetilde{\A}(x)
	\qquad
	\text{for \(\hh\)-a.e.\ } x\in\Omega.
	\end{equation}
	From Theorem~\ref{t:pairing}, \eqref{f:Akd} and (iii) we have that
	\begin{gather*}
	\A_k(x) \cdot D^d u \to \widetilde{\A}(x)\cdot D^d u,
	\qquad \text{\(|D^d u|\)-a.e.\ in}\ \Omega,\\
	\Trace[]{\A_k}{J_u}(u^+ - u^-)
	\to \Trace[*]{\A}{J_u}(u^+ - u^-),
	\qquad
	\text{\(\hh\)-a.e.\ in}\ J_u,
	\end{gather*}
	hence
	\[
	(\A_k, Du)^d \to (\A, Du)^d,
	\qquad
	(\A_k, Du)^j \to (\A, Du)^j\,.
	\]
	
	(v)
	Using the definition \eqref{theta} of \(\theta\),
	we have that, for every \(\varphi\in C_c(\Omega)\),
	\[
	\begin{split}
	\int_\Omega \theta(\A_k, Du, x) \varphi(x)\, d|Du|
	& =
	\pscal{(\A_k, Du)}{\varphi}
	\\ & \to
	\pscal{(\A, Du)}{\varphi}
	=
	\int_\Omega \theta(\A, Du, x) \varphi(x)\, d|Du|,
	\end{split}
	\]
	hence (v) follows.
\end{proof}

\begin{proposition}\label{approx2}
Let $(\A_k)$ be a sequence in $\DM(\Omega)$ such that $\A_k\to\A\in\DM(\Omega)$ in $L^1_{\rm loc}(\Omega;\R^N)$
and the sequence $\mu_k := \Div \A_k$ locally weakly${}^*$ converges to \(\mu := \Div\A\).
Let \(u\in BV(\Omega)\cap L^\infty(\Omega)\)
be compactly supported in \(\Omega\).
Then the following hold:
\begin{itemize}
\item[(a)]
If the measures \(\mu_h\) are positive and \(u\geq 0\),
then
\begin{gather}
\int_\Omega  u^-\, d\mu \leq \liminf_{h\to\infty}\int_\Omega u^-\, d\mu_h,
\label{eq1}\\
\int_\Omega  u^+\, d\mu \geq \limsup_{h\to\infty}\int_\Omega u^+\, d\mu_h,
\label{eq2}
\end{gather}
where \(u^-\) (resp.\ \(u^+\)) is the approximate lower (resp.\ upper) limit of \(u\).
 
\item[(b)]
Assume that \(|\mu_h| \stackrel{*}{\rightharpoonup} |\mu|\)
locally weakly${}^*$. 
If 
\(|\mu|(J_u) = 0\), then
\begin{equation}\label{eq3}
\int_\Omega u^*\, d\mu = \lim_{h\to +\infty} \int_\Omega u^*\, d\mu_h,
\qquad
\int_\Omega u^\pm\, d\mu = \lim_{h\to +\infty} \int_\Omega u^\pm\, d\mu_h.
\end{equation}
\end{itemize}	
\end{proposition}

\begin{proof}
(a)
Let us first consider the case \(u\in W^{1,1}(\Omega)\cap L^\infty(\Omega)\).
Since 
$u$ has compact support in \(\Omega\),
it follows that 
\begin{equation}\label{aaa}
\int_\Omega  \widetilde u\, d\mu 
=
-\int_\Omega \nabla u \cdot \A \, dx
=
\lim_{k\to\infty}-\int_\Omega  \nabla u\cdot \A_k\, dx
=
\lim_{k\to\infty}\int_\Omega  \widetilde u\, d\mu_k.
\end{equation}

Let us now consider the general case \(u\in BV(\Omega)\).
From Theorem~\ref{t:lecce},
the approximate upper limit $u^+$ and the approximate
lower limit $u^-$ are $C_1$-quasi upper semicontinuous and $C_1$-quasi lower semicontinuous,
respectively.
In order to prove \eqref{eq1}, we remark that  by Lemma \ref{maso1}
there exists an increasing sequence of nonnegative
functions $(u_h)\subseteq W^{1,1}(\Omega)$ such that, for every $h\in\N$,
$u_h$ is approximately continuous $\H^{N-1}$-almost everywhere in $\Omega$ and
$\ut_h(x)\to u^-(x)$, when $h\to+\infty$,  for $\H^{N-1}$-almost every $x\in\Omega$.

Therefore for $\H^{N-1}$-almost every $x\in\Omega$
\[
u^-(x)=\sup_{h\in\N}\widetilde u_h(x)
\]
and for every $\phi\in C^0_c(\Omega)$, with $0\leq\phi\leq 1$, we have
\[
\int_\Omega \phi u^-\, d\mu =\sup_{h\in\N}\int_\Omega \phi\widetilde u_h \, d\mu.
\]
Moreover, since $u\in L^\infty(\Omega)$, we can assume that,
for every \(h\in\N\), $u_h\in L^\infty(\Omega)$, then
$\phi\, u_h\in W^{1,1}(\Omega)\cap L^\infty(\Omega)$, with compact support,  
and  $\mu(S_{\phi u_h})=0$. 
Hence, by \eqref{aaa},
\[
\int_\Omega \phi\, \widetilde u_h \, d\mu 
= \lim_{k\to\infty}\int_\Omega \phi\, \widetilde u_h\, d\mu_k
\leq
\liminf_{k\to\infty}\int_\Omega \phi\, u^-\, d\mu_k.
\]
The conclusion follows taking the supremum among all the functions $\phi\in C^0_c(\Omega)$, 
with $0\leq\phi\leq 1$, and among the 
$h\in\N$.

The proof of \eqref{eq2} is similar, since by Lemma \ref{maso1}
there exists a decreasing sequence of nonnegative
functions $(v_h)\subseteq W^{1,1}(\Omega)$ such that, for every $h\in\N$,
$v_h$ is approximately continuous $\H^{N-1}$-almost everywhere in $\Omega$ and
$\widetilde{v}_h(x)\to u^+(x)$, when $h\to+\infty$,  for $\H^{N-1}$-almost every $x\in\Omega$.
Therefore for $\H^{N-1}$-almost every $x\in\Omega$
\[
u^+(x)=\inf_{h\in\N}\widetilde v_h(x)
\]
and  we have
\[
\int_\Omega u^+\, d\mu = \inf_{h\in\N}\int_\Omega\widetilde v_h \, d\mu.
\]
Moreover, since $u\in L^\infty(\Omega)$, we have that $v_h\in L^\infty(\Omega)$ for any $h$  sufficiently large, and since the support of $u$ is compact and $u\in L^\infty(\Omega)$ there exists a relatively compact 
neighborhood $U$ of the support of $u$ which contains the support of $v_h$ for any $h$  sufficiently large. Therefore 
$v_h\in W^{1,1}(\Omega)\cap L^\infty(\Omega)$ and it has compact support for $h$  sufficiently large,  and  $\mu(S_{v_h})=0$. 
Hence  we get
\[
\int_\Omega \widetilde v_h \, d\mu
= \lim_{k\to\infty}\int_\Omega \widetilde v_h\, d\mu_k
\geq
\limsup_{k\to\infty}\int_\Omega v^+\, d\mu_k.
\]
The conclusion follows taking the infimum among the 
$h\in\N$.

(b)
In order to prove \eqref{eq3} firstly we assume that $\mu_k\geq 0$. 
We observe that $\widetilde{v}_h-\ut_h\to 0$ $\H^{N-1}$-a.e.\ on $\Omega\setminus S_u$
and, since $\mu(S_u)=0$, 
\[
\lim_{h\to +\infty}\int_\Omega (\widetilde v_h-\widetilde u_h)\, d\mu=0.
\]
We have 
\[
\begin{split}
\int_\Omega \widetilde u_h\, d\mu 
& =
\lim_{k\to\infty}\int_\Omega \widetilde u_h\, d\mu_k
\leq
\liminf_{k\to\infty}\int_\Omega u^-\, d\mu_k
\leq
\limsup_{k\to\infty}\int_\Omega u^+\, d\mu_k
\\ &
\leq
\lim_{k\to\infty}\int_\Omega \widetilde v_h\, d\mu_k
=
\int_\Omega \widetilde v_h\, d\mu.
\end{split}
\]
By taking $h\to+\infty$, we obtain that
\[
\int_\Omega u^-\, d\mu
=
\lim_{k\to\infty}\int_\Omega u^-\, d\mu_k
=
\lim_{k\to\infty}\int_\Omega u^+\, d\mu_k
=
\int_\Omega u^+\, d\mu.
\]
By the definition of $u^*$ we get
\[
\lim_{k\to\infty}\int_\Omega u^-\, d\mu_k
=
\lim_{k\to\infty}\int_\Omega u^+\, d\mu_k
=
\lim_{k\to\infty}\int_\Omega u^*\, d\mu_k.
\]
The general case can we obtained by writing the measure $\mu$ as the difference between its positive and its negative part.
This concludes the proof.
\end{proof}

\begin{remark}
We would like to underline two consequences
of Proposition~\ref{approx2}.

(a) By \eqref{eq3}, for every \(u\in\BVLloc\), if $|\Div\A|(J_u)=0$, then
$$
\pscal{(\A_k, Du)}{\phi}\to \pscal{(\A, Du)}{\phi}\qquad \forall \phi\in C^0_c(\Omega)\,.
$$

(b) If $-\Div \A_k\geq 0$, then
\[
-\int_\Omega  u^-\Div\A\leq\liminf_{k\to\infty}\left(-\int_\Omega u^-\Div \A_k\right).
\]
\end{remark}

\section{The Gauss--Green formula}
\label{s:GG}

In this section we will prove a generalized Gauss--Green formula
for vector fields \(\A\in \DMloc[\R^N]\) on a set \(E\subset\R^N\) of finite perimeter.

Using the conventions of Section~\ref{distrtraces},
we will assume that the generalized normal vector on \(\partial^* E\) coincides
\(\hh\)-a.e.\ on \(\partial^* E\) with the measure--theoretic 
\textsl{interior} unit normal vector \(\nuint_E\) to \(E\).
Hence, if $\alpha^\pm := \Trace{\A}{\partial^* E}$
are the normal traces of \(\A\) on \(\partial^* E\)
according to our definition in Section~\ref{distrtraces},
then, using the notation of \cite{ComiPayne}, 
\(\alpha^+ \equiv (\mathcal{A}_i\cdot\nuint_E)\) and \(\alpha^- \equiv (\mathcal{A}_e\cdot\nuint_E)\) 
correspond respectively to the interior and the exterior normal traces on \(\partial^* E\).

Since \(|D\chi_E| = \hh\res\partial^* E\), 
from Proposition~\ref{p:uA} we deduce that \(\alpha^+\) and \(\alpha^-\) are respectively
the Radon--Nikod\'ym derivatives with respect to \(|D\chi_E|\)
of the measures
\[
\sigma_i := 2\, (\chi_E \A, D\chi_E),
\qquad
\sigma_e := 2\, (\chi_{\R^N\setminus E} \A, D\chi_E),
\]
that are both absolutely continuous with respect to \(|D\chi_E|\), hence
\[
\sigma_i = \alpha^+\, \hh\res\partial^* E,
\qquad
\sigma_e = \alpha^-\, \hh\res\partial^* E
\]
(see also \cite[Theorem 3.2]{ComiPayne}).

For example, if \(E\) is an open bounded set of class \(C^1\) and \(\A\) is a piecewise continuous vector field
that can be extended continuously by vector fields \(\A_i\) and \(\A_e\)
in \(\overline{E}\) and \(\R^N\setminus E\) respectively, then
\[
\alpha^+ = -\Trace[]{\A}{\partial E} = 
- \A_i \cdot \nu_E = \A_i\cdot\nuint_E, 
\qquad
\alpha^- = \A_e \cdot \nuint_E.
\]

If \(u\in BV_{{\rm loc}}(\R^N)\), in the following formulas we denote
\[
u^\pm(x) := \ut(x)
\qquad \forall x\in \R^N\setminus S_u.
\]

\begin{theorem}\label{t:GG}
Let $\A \in \DMloc[\R^N]$, $u\in BV_{{\rm loc}}(\R^N)$ and assume that $u^*\in L^1_{\rm{loc}}(\R^N,\Div \A)$.
Let \(E \subset \R^N\) be a bounded set with finite perimeter.
Then the
following Gauss--Green formulas hold:
\begin{gather}
\int_{E^1} u^* \, d\Div\A + \int_{E^1} (\A, Du) = -
\int_{\partial ^*E} \alpha^+u^+ \ d\mathcal H^{N-1}\,,\label{GreenIB}
\\
\int_{E^1\cup \partial ^*E} u^* \, d\Div\A + \int_{E^1\cup \partial ^*E} (\A, Du) = -
\int_{\partial ^*E} \alpha^-u^- \ d\mathcal H^{N-1}\,,\label{GreenIB2}
\end{gather}
where $E^1$ is the measure theoretic interior of $E$ and 
$\alpha^\pm := \Trace{\A}{\partial^* E}$
are the normal traces of \(\A\) when \(\partial^* E\) is oriented
with respect to the interior unit normal vector. 
\end{theorem}

\begin{remark}
We emphasize that the assumptions on $\A$ and $u$ in
Theorem~\ref{t:GG} are in some sense minimal.
Namely, on the vector field $\A$ we require the minimal regularity
in order to have $\Div \A$ a measure
and to have the existence of
weak normal traces along countably $\H^{N-1}$-rectifiable sets. 
Moreover, the class $BV$ for the function $u$ 
is required to construct the pairing measure,
and it is enough to guarantee
the existence of traces on these sets.
In other words, our feeling is that we cannot weaken
any assumptions on $\A$ or $u$ without losing the meaning of
at least one ingredient in formulas \eqref{GreenIB}--\eqref{GreenIB2}.

It is also worth to underline that assumption~\eqref{f:ipo} is not
required here,
since it is needed only to prove the relation
\((\A, Du)^d = \widetilde{\A} \cdot D^d u\). 
Namely, the proof of 
\eqref{GreenIB}--\eqref{GreenIB2}
is based on two main ingredients:
(i)
the characterization of the weak normal traces of
$\Div(u\A)$ on $\partial^* E$ given by
Proposition~\ref{p:tracesb},
and 
(ii) the Leibniz formula stated in Proposition~\ref{p:uA} 
which is a consequence of \eqref{thetacomp}
and, at the end, of the coarea formula 
\eqref{f:thetaut}.
\end{remark}

\begin{remark}
This result extends Theorem 5.3 of \cite{ChToZi} where $u=\phi\in C^\infty_c$
 (see also \cite[Theorem 4.1]{ComiPayne} where $u=\phi\in {\rm {Lip}}_{\rm {loc}}$).
Leonardi and Saracco (see Theorem 2.2 in \cite{LeoSar}) established a similar formula by considering the collection $X(E)$ of vector fields $\A\in L^\infty(E;\R^N) \cap C^0(E;\R^N)$ such
that $\Div\A \in L^\infty(E)$ and by assuming that the set $E$ with finite perimeter satisfies a weak regularity condition.
(We remark that, in this case, there is the additional difficulty that the vector field $\A$ is defined only
on $E$.) 
\end{remark}

\begin{proof}
Since \(E\) is bounded, without loss of generality we can assume that
\(\A\in\DM(\R^N)\) and \(u\in BV(\R^N)\).
We divide the proof into two steps.

\noindent
\textsl{Step 1}. 
Firstly, we consider the case $u\in L^\infty(\R^N)$.
Since $E$ is a bounded set with finite perimeter, we have that $\chi_E\in BV(\R^N)$ 
and the reduced boundary $\partial ^*E$ is a $\hh$-rectifiable set. 
Moreover, the vector field \(\chi_E u \A\) is compactly supported,
so that 
\[
\Div(\chi_E u \A) (\R^N) = 0
\]
(see \cite[Lemma 3.1]{ComiPayne}).
Hence by choosing in \eqref{f:anz} $\chi_E$ instead of $u$ and $u\A$ instead of $\A$, we get
\begin{equation}\label{GreenIIIA}
\int_{\R^N} \chi_E^* \, d\Div(u\A) =
- (u\A, D\chi_E) (\R^N).
\end{equation}
We recall that 
\[
\chi_E^*=\chi_{E^1}+\frac12\chi_{\partial ^*E},
\]
and, by Proposition \ref{p:tracesb} and the definition of normal traces 
it holds
\[
\Div (u\A) \res \partial^* E =
(u^+ \alpha^+ - u^- \alpha^-) \hh \res \partial^* E\,.
\]
Hence
\begin{equation}\label{f:gl}
\int_{\R^N} \chi_E^* \, d\Div(u\A)  =
\int_{E^1} \, d\Div(u\A) +\frac12\int_{\partial ^*E} \, [u^+\alpha^+-u^-\alpha^-] \ d\mathcal H^{N-1}\,.
\end{equation}
On the other hand $D\chi_E = \nuint_E\, \hh \res \partial^* E$
so that, by Proposition~\ref{p:uA},
\[
(u\A,D\chi_E)= (u\alpha)^*\, \hh\res \partial^* E,
\]
that in turn gives
\begin{equation}\label{f:gr}
(u\A, D\chi_E) (\R^N) = 
\int_{\partial ^*E}\frac12 [u^+\alpha^++u^-\alpha^-]\ d\mathcal H^{N-1}\,.
\end{equation}
Finally, substituting \eqref{f:gl} and \eqref{f:gr} in \eqref{GreenIIIA} and simplifying,
we obtain \eqref{GreenIB}.

On the other hand,
\[
\begin{split}
\int_{E^1\cup \partial ^*E} \, d\Div(u\A)=&\int_{E^1} \, d\Div(u\A) +\int_{\partial ^*E} \, [u^+\alpha^+-u^-\alpha^-] \ d\mathcal H^{N-1}
\\
=&-\int_{\partial ^*E}u^+\alpha^+\ d\mathcal H^{N-1}+\int_{\partial ^*E} \, [u^+\alpha^+-u^-\alpha^-] \ d\mathcal H^{N-1}\,,
\end{split}
\]
hence \eqref{GreenIB2} follows.
This concludes the proof of  Step 1.

\smallskip
\noindent
\textsl{Step 2}. 
Let us consider now $u\in BV(\R^N)$ such that $u^*\in L^1_{\rm{loc}}(\R^N,\Div \A)$. 
As in the proof of Theorem~\ref{t:pairing2},
let \(u_k := T_k(u)\) be the truncated functions of \(u\),
where \(T_k\) is the truncation operator defined in \eqref{trun}.

By Step 1, since $T_k(u)\in L^\infty(\R^N)$ we obtain
\begin{equation}\label{ec:4}
\int_{E^1} T_k(u)^* \, d\Div\A + \int_{E^1} (\A, DT_k(u)) =
-\int_{\partial^*E} \alpha^+ T_k(u^+) \ d\mathcal H^{N-1}\,,
\end{equation}
for every $k>0$.  
We have that
\[
T_k(u)^* = \frac{T_k(u)^+ + T_k(u)^-}{2} = \frac{T_k(u^+) + T_k(u^-)}{2}
\to \frac{u^+ + u^-}{2} = u^*,
\qquad
\text{\(\hh\)-a.e.},
\]
hence \(T_k(u)^*(x) \to u^*(x)\) for \(|\Div\A|\)-a.e.\ \(x\in\R^N\).   
Since \(|T_k(u)^*| \leq |u^*| \in L^1_{{\rm loc}}(\R^N, \Div \A)\),
from the Dominated Convergence Theorem we have that
\begin{equation}\label{f:Tk}
\lim_{k\to +\infty}
\int_{E^1} T_k(u)^* \, d\Div\A = \int_{E^1} u^*\, d\Div\A.
\end{equation}
With a similar argument we also get that
\begin{equation}\label{f:Tkk}
\lim_{k\to +\infty}
\int_{\partial^* E} \alpha^+ T_k(u)^+\, d\hh = \int_{\partial^* E} \alpha^+ u^+\, d\hh.
\end{equation}
On the other hand, by the definition \eqref{f:anz} of pairing,
for every \(\varphi\in C^\infty_c(\R^N\) it holds
\[
\langle(\A, DT_k(u)),\varphi\rangle=-\int_{\R^N} T_k(u)^*\varphi
\, d\Div\A-\int_{\R^N} T_k(u)\A\cdot\nabla\varphi\, dx\,.
\]
We can use the Dominated Convergence Theorem 
in both integrals at the right--hand side 
(for the first one we can reason as in \eqref{f:Tk}),
obtaining 
\begin{equation}\label{f:Ak}
\lim_{k\to\infty}\int_{E^1} (\A, DT_k(u)) =\int_{E^1} (\A, Du)\,.
\end{equation}
Finally, from \eqref{ec:4}, \eqref{f:Tk}, \eqref{f:Tkk} and \eqref{f:Ak} 
we get \eqref{GreenIB}.
Formula \eqref{GreenIB2} can be obtained in a similar way.
\end{proof}


\taglio{
\begin{proof}
Since, for \(\mathcal{L}^1\)-a.e.\ \(t\in\R\),
\[
\frac{d Du}{d|Du|} = \frac{d D\chiut}{d|D\chiut|}
\qquad
\text{\(|D\chiut|\)-a.e.\ in}\ \Omega,
\]
we have that, for \(\mathcal{L}^1\)-a.e.\ \(t\in\R\),
\[
\theta(\A, Du, x) = \theta(\A, D\chiut, x),
\qquad
\text{for \(|D\chiut|\)-a.e.}\ x\in\Omega.
\]
Let \(\varphi\in C_c(\Omega)\).
Since the function \(x\mapsto \theta(\A, Du, x) \varphi(x)\) is
\(|Du|\)-summable, from the coarea formula in \(BV\)
(see \cite[Theorem~3.40]{AFP}) it holds
\[
\begin{split}
\pscal{(\A, Du)}{\varphi} & = \int_\Omega \theta(\A, Du, x)\varphi\, d |Du|
= \int_\R dt \int_\Omega
\theta(A, Du, x)\varphi\, d|D\chiut|
\\ & = \int_R dt \int_\Omega \theta(\A, D\chiut, x)\varphi\, d|D\chiut|
= \int_\R \pscal{(\A, D\chiut)}{\varphi}\, dt,
\end{split}
\]
completing the proof.
\end{proof}
} 

\bigskip
\textsc{Acknowledgments.}
The authors would like to thank 
Giovanni E.\ Comi and Gian Paolo Leonardi
for some useful remarks
on a preliminary version of the manuscript.

They also want to thank the anonymous referee for carefully reading the manuscript and for giving constructive comments which helped improving the quality of the paper.

\def\cprime{$'$}

\end{document}